\documentclass[11pt]{article}
\usepackage{epigamath}

\usepackage[notext]{kpfonts}
\usepackage{baskervald}

\setpapertype{A4}

\usepackage[french]{babel}

\usepackage[dvips]{graphicx}     

\removelength{1cm}

\title{Sur l'hyperbolicit\'e de graphes associ\'es au groupe de Cremona}
\titlemark{Sur l'hyperbolicit\'e de graphes associ\'es au groupe de Cremona}
\author{Anne Lonjou}
\authoraddresses{
\authordata{Anne Lonjou}{\firstname{Anne} \lastname{Lonjou}\\
\institution{Universit\"at Basel, Departement Mathematik und
Informatik, Spiegelgasse 1,
4051 Basel, Switzerland}\\
\email{anne.lonjou@unibas.ch}} }
\authormark{A. Lonjou}
\date{\vspace{-5ex}} 
\journal{\'Epijournal de G\'eom\'etrie Alg\'ebrique} 
\acceptation{Received by the Editors on October 18, 2018, and in
final form on April 8, 2019. \\ Accepted on June 28, 2019.}

\newcommand{\articlereference}{Volume 3 (2019), Article Nr. 11}

\usepackage[all]{xy}

\allowdisplaybreaks

\newdimen\origiwspc
\origiwspc=\fontdimen2\font

\usepackage{float}

\numberwithin{equation}{section}

\makeatletter

\renewcommand{\p@equation}{\arabic{section}.\arabic{equation}\expandafter\@gobble}
\makeatother

\usepackage{enumitem}
\setlist[enumerate,1]{label={\rm \arabic*)}}

\usepackage{tikz-cd}

\usepackage{tikz}
\usetikzlibrary{decorations.pathmorphing, decorations.markings, arrows, matrix, fit, positioning}
\tikzset{
    typeone/.style={circle,draw=black,very thick,inner sep=0pt,minimum size=7mm},
    typetwo/.style={circle,draw=black,fill=black,thin,inner sep=0pt,minimum size=3mm},
    typethree/.style={rectangle,draw=black,fill=black,thin,inner sep=0pt,minimum size=6mm}
} \tikzstyle directed=[postaction={decorate,decoration={markings,
mark=at position .55 with {\arrow{stealth}}}}]

\newtheorem{thm}[equation]{Th\'eor\`eme}
\newtheorem{prop}[equation]{Proposition}
\newtheorem{lemme}[equation]{Lemme}
\newtheorem{cor}[equation]{Corollaire}

\newtheorem{maintheorem}{Th\'eor\`eme}

\newtheorem{defi}[equation]{D\'efinition}
\newtheorem{rmq}[equation]{Remarque}
\newtheorem{fait}[equation]{Fait}

\newcommand{\B}{\mathcal{B}}

\renewcommand{\D}{\mathcal{A}}
\newcommand{\F}{\mathbb{F}}
\newcommand{\Fl}{F}
\newcommand{\G}{\mathcal{G}}
\newcommand{\GC}{\mathcal{GC}}

\renewcommand{\j}{\mathfrak{j}}

\newcommand{\N}{\mathbb{N}}
\newcommand{\Ne}{\mathcal{N}}

\renewcommand{\P}{\mathbb{P}}
\newcommand{\q}{\mathfrak{q}}
\newcommand{\R}{\mathbb{R}}

\newcommand{\V}{\mathcal{V}}
\newcommand{\Z}{\mathbb{Z}}

\newcommand{\W}{\mathcal{W}}

\newcommand{\isome}[1]{#1_{{\scriptscriptstyle{\#}}}}
\newcommand{\orbl}{\mathcal{E}}
\newcommand{\Som}{\mathcal{S}}

\renewcommand{\H}{\mathbb{H}^{\infty}}

\renewcommand{\epsilon}{\varepsilon}
\renewcommand{\l}{\ell}

 \DeclareMathOperator{\Bir}{Bir}

\DeclareMathOperator{\PGL}{PGL} \DeclareMathOperator{\PSL}{PSL}
 \DeclareMathOperator{\CAT}{CAT}
\DeclareMathOperator{\PM}{\mathcal{Z}} \DeclareMathOperator{\Bs}{Bs}
 \DeclareMathOperator{\Mod}{Mod}
\DeclareMathOperator{\NS}{N^1} \DeclareMathOperator{\kk}{k}
\DeclareMathOperator{\can}{K} 
\DeclareMathOperator{\card}{Card}

\DeclareMathOperator{\diam}{Diam} 
\DeclareMathOperator{\dist}{d} 
 
 \DeclareMathOperator{\id}{id}
\DeclareMathOperator{\jonq}{j_{onq}} \DeclareMathOperator{\Out}{Out}
 
\DeclareMathOperator{\lgueur}{\l g} 
 \DeclareMathOperator{\md}{\#_{md}}
 
\DeclareMathOperator{\argcosh}{argcosh}

\DeclareMathOperator{\supp}{supp} 
 \DeclareMathOperator{\Pic}{Pic}

\begin{document}


\maketitle



\begin{prelims}

\vspace{-0.55cm}

\def\abstractname{R\'esum\'e}
\abstract{Poursuivant l'analogie d\'ej\`a existante entre le groupe modulaire et le groupe de Cremona de rang $2$ sur un corps
alg\'ebriquement clos, nous cherchons un graphe sur lequel le groupe
de Cremona agit de fa\c{c}on non triviale et qui est un analogue du
graphe des courbes. Un candidat naturel est un graphe introduit par
D. Wright. Cependant, r\'epondant du m\^eme coup \`a une question de
A. Minasyan et D. Osin, nous montrons que ce graphe n'est pas
Gromov-hyperbolique. Nous construisons ensuite deux graphes
associ\'es au pavage de Vorono\"{\i} du groupe de Cremona, introduit
dans un pr\'ec\'edent travail de l'auteur. Nous montrons que l'un
est quasi-isom\'etrique au graphe de Wright. Nous prouvons que le
second, quant \`a lui, est hyperbolique.}

\motscles{Groupe de Cremona, espace hyperbolique, graphe
Gromov-hyperbolique, graphe de Cayley, pavage de Vorono\"{\i}}

\MSCclassfr{14E07; 20F65}


\languagesection{English}{%

\vspace{-0.05cm} \textbf{Title. On the hyperbolicity of graphs
associated to the Cremona group} \commentskip \textbf{Abstract.} To
reinforce the analogy between the mapping class group and the
Cremona group of rank $2$ over an algebraic closed field, we look
for a graph analogous to the curve graph and such that the Cremona
group acts on it non-trivially. A candidate is a graph introduced by
D. Wright. However, we demonstrate that it is not Gromov-hyperbolic.
This answers a question of A. Minasyan and D. Osin. Then, we
construct two graphs associated to a Vorono\"{\i} tessellation of the
Cremona group introduced in a previous work of the author. We show
that one is quasi-isometric to the Wright graph. We prove that the
second one is Gromov-hyperbolic.}

\end{prelims}


\newpage

\setcounter{tocdepth}{1} \tableofcontents

\bigskip

\section*{Introduction}
\addcontentsline{toc}{section}{Introduction} Dans cet article, le
corps de base, not\'e $\kk$, est alg\'ebriquement clos (sauf mention
du contraire). Les surfaces consid\'er\'ees sont projectives et
lisses. Les graphes \'etudi\'es sont munis de la m\'etrique standard
qui rend chaque ar\^ete isom\'etrique \`a l'intervalle ferm\'e
r\'eel $[0,1]$.

Nous nous int\'eressons au groupe de Cremona de rang $2$ sur un
corps $\kk$, not\'e $\Bir(\P^2)$, qui est le groupe des
transformations birationnelles du plan projectif (pour all\'eger
l'\'ecriture nous notons le plan projectif $\P^2$ au lieu de
$\P^2_{\kk}$). Il existe des analogies entre le groupe de Cremona,
le groupe $\PSL(2,\Z)$ et le groupe modulaire $\Mod(\Sigma_g)$ d'une
surface compacte de genre $g\geq 2$ (le groupe des hom\'eomorphismes
de cette surface pr\'eservant l'orientation \`a homotopie pr\`es).
Les groupes $\PSL(2,\Z)$ et $\Mod(\Sigma_g)$ agissent sur des
espaces combinatoires qui sont Gromov-hyperboliques, c'est-\`a-dire
des espaces combinatoires dont les triangles sont uniform\'ement
fins. Le groupe $\PSL(2,\Z)$ agit sur l'arbre de Bass-Serre
associ\'e au scindement $\PSL(2,\Z)=\Z/2\Z \ast \Z/3\Z$. Le groupe
$\Mod(\Sigma_g)$ agit sur le complexe des courbes dont les sommets
sont les classes d'isotopies de lacets simples et essentiels sur
$\Sigma_g$ et dont les simplexes de dimension $n$ sont donn\'es par
$n+1$ classes d'isotopies de lacets simples et essentiels de
$\Sigma_g$ admettant des repr\'esentants deux \`a deux disjoints.
Ces deux groupes ont \'egalement \'et\'e une source d'inspiration
pour l'\'etude du groupe $\Out(\Fl_n)$, le groupe des automorphismes
ext\'erieurs du groupe libre \`a $n$ g\'en\'erateurs (le quotient du
groupe des automorphismes de $\Fl_n$ par le sous-groupe distingu\'e
des automorphismes int\'erieurs de $\Fl_n$). Ainsi, des complexes
simpliciaux sur lesquels le groupe $\Out(\Fl_n)$ agit de fa\c{c}on
non triviale et \'etant candidats \`a \^etre analogue au complexe
des courbes, ont \'et\'e construits (voir
\cite{Best_Feighn_free_factors} et \cite{HanMosh}).

Nous nous proposons de faire de m\^eme pour le groupe de Cremona. Le
but de cet article est de construire un graphe Gromov-hyperbolique
sur lequel le groupe de Cremona agit de fa\c{c}on non triviale.
\subsubsection*{Graphe de Wright}

Le complexe de Wright \cite{W} est un complexe simplicial, de
dimension $2$ et simplement connexe sur lequel le groupe de Cremona
agit. Il est obtenu en utilisant la structure de produit amalgam\'e
sur trois facteurs du groupe de Cremona sur un corps
alg\'ebriquement clos. Comme nous nous int\'eressons \`a la
propri\'et\'e de Gromov-hyperbolicit\'e, seul le $1$-squelette nous
int\'eresse. Nous montrons que ce graphe non localement fini est de
diam\`etre infini (Corollaire \ref{cor_Wright_infini}), ce qui
n'\'etait pas \'evident a priori. Nous nous demandons ensuite si
ce graphe est Gromov-hyperbolique. Cette question a \'et\'e pos\'ee
par A. Minasyan et D. Osin dans \cite[Problem 8.5]{MO}. Leur
motivation \'etait, dans le cas o\`u la r\'eponse est positive, de
trouver une nouvelle mani\`ere de montrer que le groupe de Cremona
n'est pas simple en utilisant les r\'esultats de \cite{DGO}. En
fait, la r\'eponse \`a cette question est n\'egative.
\begin{maintheorem}\label{main_thm_Wright_nonhyperbolique}
Le graphe de Wright n'est pas hyperbolique au sens de Gromov.
\end{maintheorem}

Le premier point dans la preuve est de remarquer que le graphe de
Wright est quasi-isom\'etrique \`a un graphe li\'e au syst\`eme de
g\'en\'erateurs du groupe de Cremona donn\'e par $\PGL(3,\kk)$ et
les applications de Jonqui\`eres. C'est un analogue du graphe de
Cayley dans le cas d'un groupe de type fini. Nous appelons ce graphe
\og le graphe de Wright modifi\'e \fg{}. Les sommets de ce graphe
sont les \'el\'ements du groupe de Cremona modulo pr\'e-composition
par un \'el\'ement de $\PGL(3,\kk)$. Une ar\^ete relie deux sommets
s'il existe une application de Jonqui\`eres envoyant un sommet sur
l'autre. Ici, une application de Jonqui\`eres est une application
pr\'eservant un pinceau de droites. La distance entre deux sommets
$f,g\in\Bir(\P^2)$ correspond au nombre minimal d'applications de
Jonqui\`eres qu'il faut pour d\'ecomposer l'application $g^{-1}\circ
f$. Dans \cite{BlF}, J. Blanc et J-P. Furter appellent cet entier
\og longueur de translation \fg{} de l'application $g^{-1}\circ f$.
Ils d\'eterminent un algorithme pour d\'eterminer cette longueur.
\`A noter qu'ils obtiennent eux aussi que le graphe de Wright est de
diam\`etre infini. Cependant la preuve pr\'esent\'ee ici est plus
\'el\'ementaire car elle ne requiert pas la notion de longueur d'une
application. Une cons\'equence de leur th\'eor\`eme principal
(\cite[Corollaire 3.29]{BlF}) est que l'algorithme propos\'e dans
\cite{AC} donne aussi le nombre minimal de Jonqui\`eres.

Le second point est de montrer que ce graphe contient un sous-graphe
quasi-isom\'etrique \`a $\Z^2$ (Th\'eor\`eme \ref{thm_qi_gc_z2}).
Nous utilisons pour cela deux \'el\'ements du groupe de Cremona,
appel\'es twists de Halphen, qui sont d'ordre infini et qui
commutent. Ils engendrent un sous-groupe isomorphe \`a $\Z^2$. Il
reste encore \`a montrer que l'action de ce sous-groupe sur un des
sommets du graphe induit le sous-graphe recherch\'e. Nous utilisons
pour cela des r\'esultats de J. Blanc et J-P. Furter. \`A noter que
le graphe de Cayley du groupe $\Mod(\Sigma_g)$ n'est pas
Gromov-hyperbolique car ce groupe poss\`ede \'egalement des
sous-groupes isomorphes \`a $\Z^2$, engendr\'es par exemple par deux
twists de Dehn le long de deux lacets disjoints.

\subsubsection*{Graphes associ\'es au pavage de Vorono\"{\i}}
Un espace tr\`es utile dans l'\'etude du groupe de Cremona est
l'espace hyperbolique de dimension infinie, not\'e $\H$, qui est un
analogue du mod\`ele de l'hyperbolo\"{\i}de de l'espace hyperbolique
de dimension $n$. Il vit dans l'espace de Picard-Manin qui est le
compl\'et\'e $\l^2$ de la limite inductive des groupes de
N\'eron-Severi, \`a coefficients r\'eels, des surfaces dominant
$\P^2$. L'espace de Picard-Manin est muni d'une forme bilin\'eaire
de signature $(1,\infty)$ qui provient de la forme d'intersection.
L'espace $\H$ est la nappe sup\'erieure de l'hyperbolo\"{\i}de
d\'efini comme le niveau $1$ de la forme bilin\'eaire.

Les graphes \'etudi\'es dans cet article sont naturellement
associ\'es au pavage de Vorono\"{\i} construit dans un pr\'ec\'edent
article \cite{Lovoronoi1}. Nous rappelons dans un premier temps la
construction de ce pavage ainsi que les r\'esultats que nous
utilisons par la suite. Dans \cite{Lovoronoi1}, nous nous sommes
int\'eress\'es \`a un sous-domaine convexe de $\H$, not\'e $\orbl$,
et contenant l'enveloppe convexe de l'orbite de la classe de la
droite $\l\in\H$ sous l'action du groupe de Cremona. Nous avons
construit un domaine fondamental pour l'action du groupe de Cremona
modulo $\PGL(3,\kk)$ sur $\orbl$. L'outil cl\'e est la notion de
cellules de Vorono\"{\i} associ\'ees \`a cette orbite discr\`ete de
points qui d\'ecoupent l'espace $\orbl$ en zones d'influence
d\'etermin\'ees par les points de l'orbite. \`A chaque point de
l'orbite $\isome{f}(\l)$, nous associons la cellule donn\'ee par :
\[\V(f)=\{c\in\orbl\mid \text{ pour tout } g\in\Bir(\P^2),\
\dist(c,\isome{f}(\l))\leq \dist(c,\isome{g}(\l))   \}.\] Remarquons
que pour tout $a\in\PGL(3, \kk)$, $f$ et $f\circ a$ agissent
similairement sur $\l$ et par cons\'equent d\'eterminent la m\^eme
cellule de Vorono\"{\i}. Ces applications sont appel\'ees les germes
de la cellules $\V(f)$.

Nous avons \'egalement \'etudi\'e la g\'eom\'etrie de ce pavage en
d\'eterminant notamment les cellules non disjointes de la cellule
$\V(\id)$ appel\'ees \og cellules adjacentes \fg{}. Les germes de
telles cellules sont de deux types. Ceux qui sont de
\emph{caract\'eristique Jonqui\`eres}, c'est-\`a-dire les
applications du groupe de Cremona telles qu'il existe deux points
$p$ et $q$ dans $\P^2$ et qui envoient le pinceau de droites passant
par le point $p$ sur le pinceau de droites passant par le point $q$.
L'autre type de germe des cellules adjacentes \`a la cellule
$\V(\id)$ est formé des applications du groupe de Cremona qui poss\`edent
au plus huit points-base en position presque g\'en\'erale. Un
ensemble de points $\{p_0,p_1,\dots, p_r\}$ est dit \emph{en
position presque g\'en\'erale} si d'une part pour chaque $0\leq
i\leq r$ le point $p_i$ vit soit dans $\P^2$ soit dans une surface
dominant $\P^2$ qui est obtenue en \'eclatant seulement un
sous-ensemble de points de $\{p_0,\dots,p_r\}$ et si d'autre part
aucune des trois conditions suivantes n'est satisfaite : quatre des
points de cet ensemble sont align\'es, sept des points de cet
ensemble sont sur une conique, deux des points de cet ensemble sont
adh\'erents \`a un troisi\`eme point de cet ensemble (c-\`a-d ils
appartiennent \`a la transform\'ee stricte du diviseur obtenu en
\'eclatant un point de cet ensemble).

\begin{thm}[{\cite[Corollaire 4.7]{Lovoronoi1}}]\label{cor_cellule_adjacente}
L'ensemble des germes des cellules adjacentes \`a la cellule
$\V(\id)$ est constitu\'e de toutes :
\begin{enumerate}
\item[$\bullet$] les applications de caract\'eristique Jonqui\`eres,
\item[$\bullet$] les applications qui poss\`edent au plus $8$ points-base en position presque g\'en\'erale.
\end{enumerate}
\end{thm}

Dans \cite{Lovoronoi1}, les cellules partageant une classe dans le
bord \`a l'infini avec une cellule donn\'ee sont appel\'ees \og
cellules quasi-adjacentes \fg{}. Dans ce cas, le r\'esultat est :
\begin{thm}[{\cite[Corollaire 5.20]{Lovoronoi1}}]\label{cor_cellules_quasi_adjacentes}
L'ensemble des germes des cellules quasi-adjacentes \`a la cellule
$\V(\id)$ est constitu\'e de toutes :
\begin{enumerate}
\item[$\bullet$] les applications de caract\'eristique Jonqui\`eres,
\item[$\bullet$] les applications qui poss\`edent au plus $9$ points-base en position presque g\'en\'erale.
\end{enumerate}
\end{thm}

\`A partir de cette \'etude du pavage de Vorono\"{\i}, nous
construisons dans cet article un premier graphe appel\'e \og graphe
d'adjacence \fg{}. Les sommets de ce graphe sont les centres des
cellules et il y a une ar\^ete entre deux sommets lorsque les
cellules correspondantes sont adjacentes. Nous obtenons une
mani\`ere de retrouver le graphe de Wright dans $\H$ puisque nous
montrons que le graphe d'adjacence et le graphe de Wright sont
quasi-isom\'etriques (Proposition \ref{prop_qi_graphes}). La
Gromov-hyperbolicit\'e \'etant une propri\'et\'e stable par
quasi-isom\'etrie, cela implique que le graphe d'adjacence n'est pas
Gromov-hyperbolique.

Un autre graphe naturel \`a consid\'erer est le \og graphe de
quasi-adjacence\fg. Il est construit en consid\'erant la
combinatoire des cellules de Vorono\"{\i} \`a l'infini. Le graphe de
quasi-adjacence poss\`ede les m\^emes sommets que le graphe
d'adjacence. Nous mettons une ar\^ete entre deux sommets lorsque les
cellules associ\'ees sont quasi-adjacentes.

\begin{rmq}\label{rmq_adjacent_sousgraphe_quasi}
{\rm Comme deux cellules qui ont une intersection non vide
poss\`edent \'egalement une classe commune \`a l'infini
(\cite[Corollaire 5.6]{Lovoronoi1}), le graphe de quasi-adjacence
est en fait obtenu en ajoutant des ar\^etes au graphe d'adjacence.}
\end{rmq}

Nous montrons que ce graphe est toujours de diam\`etre infini et
qu'il est Gromov-hyperbolique en utilisant un crit\`ere d\^u \`a B.
Bowditch \cite{Bow} et en utilisant la Gromov-hyperbolicit\'e de
l'espace ambiant $\orbl$. \`A noter que l'espace $\H$ et le graphe
de quasi-adjacence ne sont pas quasi-isom\'etriques, tout comme le
demi-plan de Poincar\'e et l'arbre de Bass-Serre associ\'e \`a
$\PSL(2,\Z)$ ne sont pas quasi-isom\'etriques.
\begin{maintheorem}
Le graphe de quasi-adjacence est hyperbolique au sens de Gromov.
\end{maintheorem}

De ce point de vue, le graphe de quasi-adjacence est analogue \`a
l'arbre de Bass-Serre, au graphe des courbes ou encore \`a certains
complexes associ\'es au groupe $\Out(F_n)$ comme le complexe des
facteurs libres (\cite{Best_Feighn_free_factors}) ou le complexe des
scindements libres (\cite{HanMosh}). Ajouter des ar\^etes au graphe
d'adjacence a rendu Gromov-hyperbolique un graphe non
Gromov-hyperbolique. Cette construction est similaire \`a celle du
complexe des scindements libres qui peut \^etre obtenu en ajoutant
des simplexes \`a l'outre-espace, dans le cadre du groupe
$\Out(\Fl_n)$. Le complexe des scindements libres est
Gromov-hyperbolique (\cite{HanMosh}) contrairement \`a
l'outre-espace (voir par exemple \cite{Vogt}).

Un groupe $G$ est acylindriquement hyperbolique s'il existe un
espace m\'etrique Gromov-hyperbolique $(X,\dist)$ sur lequel il agit
par isom\'etries, de fa\c{c}on non \'el\'ementaire et
acylindriquement (voir par exemple \cite{Os}). L'action est \og
acylindrique \fg\  si pour tout $\epsilon>0$ et pour tous deux
points suffisamment \'eloign\'es, il existe un nombre uniform\'ement
born\'e d'\'el\'ements du groupe d\'epla\c{c}ant ces deux points
d'au plus $\epsilon$: pour tout $\epsilon>0$ il existe $R,N>0$ tel
que pour tous deux points $x,y\in X$ tels que $\dist(x,y)\geq R$ il
y a au plus $N$ \'el\'ements $g\in G$ satisfaisant $\dist(x,gx)\leq
\epsilon$ et $\dist(y,gy)\leq \epsilon$. Un crit\`ere (\cite{Os})
pour montrer qu'un groupe non virtuellement cyclique est
acylindriquement hyperbolique est de trouver un espace hyperbolique
sur lequel il agit de fa\c{c}on non-\'el\'ementaire tel qu'il existe
un \'el\'ement loxodromique satisfaisant la propri\'et\'e WPD (voir
\cite{DGO} pour plus de d\'etails). Cependant ce crit\`ere
n'implique pas que l'action en question est acylindrique mais qu'il
en existe une.

B. Bowditch \cite{Bow_tight} prouve que le graphe des courbes est
Gromov-hyperbolique et que l'action du groupe modulaire d'une
surface compacte hyperbolique sur le graphe des courbes est
acylindrique (voir \'egalement \cite[Lemma 6.49]{DGO}).

Dans \cite{Lo}, en faisant agir le groupe de Cremona sur l'espace
hyperbolique $\H$, l'auteur trouve des \'el\'ements loxodromiques
satisfaisant la propri\'et\'e WPD sur tout corps. De m\^eme, M.
Bestvina et M. Feighn \cite{Best_Feighn_free_factors} prouvent que
le complexe des facteurs libres est Gromov-hyperbolique et ils
trouvent des \'el\'ements satisfaisant la propri\'et\'e WPD. Ou
encore, M. Handel et L. Mosher \cite{HanMosh} obtiennent le m\^eme
r\'esultat en faisant agir le groupe $\Out(F_n)$ sur le complexe des
scindements libres (en prouvant que ce dernier est \'egalement
Gromov-hyperbolique). Ainsi, ces trois groupes font partie de la
famille des groupes \og acylindriquement hyperboliques \fg.

Cependant, contrairement au cas du groupe modulaire agissant sur le
complexe des courbes, l'action du groupe de Cremona (sur un corps
quelconque) sur $\H$ et l'action de $\Out(F_n)$ sur le complexe des
scindements libres ne sont pas acylindriques car tous les
\'el\'ements loxodromiques ne satisfont pas la propri\'et\'e WPD
(\cite[Remark after Definition 2.5]{Os}). Dans le cas de l'action du
groupe $\Out(F_n)$ sur le complexe des facteurs libres, tous les
\'el\'ements loxodromiques satisfont la propri\'et\'e WPD mais il
n'est pas connu si l'action est acylindrique ou pas.

M\^eme si de par sa construction le graphe de quasi-adjacence semble
plut\^ot \^etre un analogue du complexe des scindements libres que
du complexe des facteurs libres, il serait int\'eressant de voir si
l'action du groupe de Cremona (lorsque $\kk$ est alg\'ebriquement
clos) sur le graphe de quasi-adjacence est acylindrique. Un argument
en faveur de cela et que d'apr\`es \cite[Theorem 1.4]{Os} nous
savons qu'il existe un syst\`eme de g\'en\'erateurs du groupe de
Cremona tel que le graphe de Cayley soit Gromov-hyperbolique et que
l'action naturelle soit acylindrique. Or le graphe de
quasi-adjacence peut \^etre vu comme le graphe de Cayley associ\'e
\`a la famille g\'en\'eratrice constitu\'ee de $\PGL(3,\kk)$, des
applications de Jonqui\`eres et des applications ayant au plus $9$
points-base en position presque g\'en\'erale.

\`A noter que classiquement un syst\`eme de g\'en\'erateurs est
donn\'e par $\PGL(3,\kk)$ et les applications de Jonqui\`eres. Or
nous prouvons dans cet article (Corollaire
\ref{cor_wright_modif_pas_hyperbolique}) que le graphe de Cayley
associ\'e \`a cette famille, appel\'e ici \og graphe de Wright
modifi\'e \fg\ n'est pas Gromov-hyperbolique.

Cependant, prouver que l'action est acylindrique sur le graphe de
quasi-adjacence ne semble pas facile puisqu'actuellement il n'existe
pas d'algorithme calculant la distance entre deux sommets.

Apr\`es avoir introduit les notations utiles et rappel\'e des
r\'esultats que nous utiliserons dans cet article, nous
d\'efinissons dans la section \ref{section_wright} le graphe de
Wright et le graphe de Wright modifi\'e. Nous montrons qu'ils sont
quasi-isom\'etriques et de diam\`etre infini mais qu'ils ne sont pas
Gromov-hyperboliques. Dans la section
\ref{sec_graphes_associes_voronoi}, nous construisons le graphe
d'adjacence et le graphe de quasi-adjacence. Puis, nous montrons que
le premier n'est pas Gromov-hyperbolique mais que le second l'est
tout en restant de diam\`etre infini.

\section*{Remerciements}
Je remercie vivement St\'ephane Lamy, pour sa grande disponibilit\'e
et ses relectures minutieuses. Je remercie \'egalement les
rapporteurs de ma th\`ese, Charles Favre et Yves de Cornulier, pour
leurs remarques qui ont permis de rendre certains passages plus
clairs. Je remercie J. Blanc et J-P. Furter pour des discussions
concernant la longueur des applications du groupe de Cremona. Enfin,
je remercie le Fonds National Suisse de la Recherche Scientifique
qui m'a support\'e financi\`erement \`a travers le projet \og
Algebraic subgroups of the Cremona groups \fg{} 200021\_159921.

\section{Pr\'eliminaires}
Nous rappelons rapidement ici les notions dont nous avons besoin.
Pour plus de d\'etails voir \cite{Lothese}.

\subsection{Action du groupe de Cremona sur $\H$}
Un espace essentiel dans l'\'etude du groupe des transformations
birationnelles du plan projectif, $\Bir(\P^2)$, est l'espace
hyperbolique de dimension infinie vivant dans l'espace de
Picard-Manin. Plus de pr\'ecisions se trouvent dans \cite[Section
4]{BC}, \cite[Part II.4]{CL}, \cite[Section 3]{C} et \cite[Section
1.2.3]{Lothese}.

Soit $S$ une surface. Une surface $S'$ domine $S$ s'il existe un
morphisme birationnel allant de $S'$ vers $S$. Consid\'erons $S_1$
et $S_2$ deux surfaces dominant $S$ et pour $i\in\{1,2\}$,
$\pi_i:S_i\rightarrow S$ leurs morphismes respectifs vers $S$. Nous
disons que deux points $p_1\in S_1$ et $p_2\in S_2$ sont
\'equivalents si $\pi_1^{-1}\circ \pi_2$ est un isomorphisme local
sur un voisinage de $p_2$ et envoie $p_2$ sur $p_1$. L'espace des
bulles (\og Bubble space\fg\ en anglais), not\'e $\B(S)$, est
l'union de tous les points de toutes les surfaces dominant $S$
modulo cette relation d'\'equivalence.

\subsubsection{Espace de Picard-Manin}
 Nous consid\'erons le groupe de N\'eron-Severi associ\'e \`a $S$ et tensoris\'e par $\R$. Nous le notons encore $\NS(S)$. C'est donc le groupe des diviseurs \`a coefficients r\'eels sur $S$ \`a \'equivalence num\'erique pr\`es. Il est muni d'une forme bilin\'eaire sym\'etrique, la forme d'intersection. Pour tout diviseur $D$ sur $S$ nous notons $\{D\}_S$ sa classe de N\'eron-Severi ou $\{D\}$ s'il n'y a pas d'ambigu\"{\i}t\'e sur la surface. Si $\pi : S'\longrightarrow S$ est un morphisme birationnel entre deux surfaces, alors le tir\'e en arri\`ere \[\pi^* : \NS(S)\hookrightarrow \NS(S')\] qui \`a la classe d'un diviseur associe la classe de sa transform\'ee totale, est un morphisme injectif qui pr\'eserve la forme d'intersection.
        Consid\'erons la limite inductive des groupes de N\'eron-Severi des surfaces $S'$ dominant $S$ : \[\PM_C(S)=\lim\limits_{\underset{S'\rightarrow S}{\longrightarrow}}\NS(S'),\]
        o\`u l'indice $C$ fait r\'ef\'erence aux b-diviseurs de Cartier (pour plus de pr\'ecisions voir \cite{Fa}).

        Soit $E_p$ le diviseur exceptionnel obtenu lors de l'\'eclatement de la surface $S$ au point $p$. Notons $S_p$ la surface obtenue. Nous notons $e_p$ la classe du diviseur $E_p$ dans $\PM_C(S)$, c'est-\`a-dire sur toute surface dominant $S_p$, $e_p$ correspond \`a la transform\'ee totale de $E_p$ sur cette surface.
        Par la suite, nous nous int\'eressons au compl\'et\'e $L^2$ de $\PM_C(S)$, appelé l'\emph{espace de Picard-Manin} (voir \cite{CL} et \cite{C} ou encore \cite{BFJ}) :  \[\PM(S)=\{\{D_0\}_{S}+\sum\limits_{p\in\B(S)}\lambda_pe_p\mid \lambda_p\in\R,\ \sum\limits_{p\in\B(S)}\lambda_p^2<\infty \text{ et }\{D_0\}_{S} \in\NS(S)\}.\] Ses \'el\'ements sont appel\'es \og classes de Picard-Manin\fg \ ou plus simplement \og classes\fg. Les classes $e_p$ o\`u $p$ est un point de $S$ ou d'une surface dominant $S$, sont d'auto-intersection $-1$, orthogonales deux \`a deux et orthogonales \`a $\NS(S)$.
            La forme d'intersection est bien d\'efinie sur $\PM(S)$ et elle est de signature $(1,\infty)$.

            Dans le cas o\`u la surface consid\'er\'ee est $\P^2$, nous notons simplement $\PM$ l'espace de Picard-Manin associ\'e :
                    \[\PM=\{n\l+\sum\limits_{p\in\B(\P^2)}\lambda_pe_p\mid n,\lambda_p\in\R,\ \sum\limits_{p\in\B(\P^2)}\lambda_p^2<\infty \},\] o\`u $\l$ est la classe de la droite dans $\P^2$.

\subsubsection{Espace hyperbolique}

Consid\'erons l'espace \[\H(S)=\{c\in \PM(S)\mid c\cdot c=1 \text{
et } c\cdot d_0>0\},\] o\`u $d_0\in\NS(S)$ est une classe ample.
C'est un espace hyperbolique de dimension
infinie une fois muni de la distance d\'efinie par $\dist(c,c')=\argcosh(c\cdot c')$
pour tous $c,c'\in \H(S)$. Nous nous int\'eressons plus particuli\`erement \`a
$\H(\P^2)$ que nous notons $\H$. Tout \'el\'ement de $\H$ est de la
forme \[n\ell +\sum_{p\in\B(\P^2)}\lambda_pe_p \text{ o\`u } n>0
\text{ et } n^2-\sum_{p\in\B(\P^2)}\lambda_p^2=1.\]

\subsubsection{Action du groupe de Cremona sur l'espace de Picard-Manin}
Consid\'erons une r\'esolution de $f\in\Bir(\P^2)$ :
 \begin{center}
   \begin{tikzcd}[column sep=small]
       & S\arrow{dl}[above left]{\pi} \arrow{dr}{\sigma} & \\
       \P^2 \arrow[dashrightarrow]{rr}{f} & & \P^2.
   \end{tikzcd}
 \end{center}
 Le groupe $\Bir(\P^2)$ agit sur $\PM$ (et en particulier sur $\H$) via l'application $(f,c)\mapsto \isome{f}(c)$ o\`u $\isome{f}$ est d\'efinie par \[\isome{f}=\isome{\sigma}\circ (\isome{\pi})^{-1}.\]

 Plus pr\'ecis\'ement, soit $d$ le degr\'e de $f$ et notons $p_0,p_1,\dots, p_{r-1}$ ses points-base de multiplicit\'es respectives $\{m_i\}_{0\leq i\leq r-1}$, $q_0,q_1,\dots q_{r-1}$ ceux de $f^{-1}$ de multiplicit\'es $\{m_i'\}_{0\leq i\leq r-1}$ et $a_{i,j}$ le nombre d'intersection des transform\'ees totales des diviseurs exceptionnels obtenus en \'eclatant respectivement les points $p_j$ et $q_i$, dans la r\'esolution de $f$. Notons \'egalement $\Bs(f)$ l'ensemble des points base de $f$.
Consid\'erons une classe $c$ de l'espace de Picard-Manin $\PM$ :
\[c=n\l-\sum\limits_{i=0}^{r-1}\lambda_ie_{p_i}-\sum\limits_{\substack{p\in\B(\P^2)\\p\notin\Bs(f)}}\lambda_pe_{p}.\]
L'action de $f$ sur $c$ est donn\'ee par la formule :
\begin{equation}
\isome{f}(c)=
\left(nd-\sum\limits_{j=0}^{r-1}\lambda_jm_j\right)\l-\sum\limits_{i=0}^{r-1}\left(nm_i'-\sum\limits_{j=0}^{r-1}\lambda_ja_{i,j}\right)e_{q_i}-\sum\limits_{\substack{p\in\B(\P^2)\\p\notin\Bs(f)}}\lambda_p\isome{f}(e_{p}).
\label{eq_action_f_sur_c}
\end{equation}

\subsubsection{Bord \`a l'infini de l'espace hyperbolique}
L'espace $\H$ \'etant un espace m\'etrique complet $\CAT(0)$, il
existe une notion de \emph{bord \`a l'infini} qui g\'en\'eralise
celle de bord des vari\'et\'es riemanniennes de dimension finie qui
sont compl\`etes, simplement connexes et \`a courbure n\'egative ou
nulle. Plus de d\'etails se trouvent dans \cite[Chapter
II.8]{BH}, par exemple. Le bord de $\H$ peut \^etre d\'efini comme suit.
\[\partial_{\infty} \H=\{c\in \PM\mid c\cdot c=0 \text{ et } c\cdot \l>0\}.\]Un point appartenant au bord de $\H$ est parfois appel\'e \emph{un point \`a l'infini}.
Les isom\'etries de $\H$ s'\'etendent de fa\c{c}on unique en des
hom\'eomorphismes de $\H\cup\partial_{\infty}\H$.

\subsection{Pavage de Vorono\"{\i}}
Nous rappelons bri\`evement ici la construction du pavage de
Vorono\"{\i} que nous avons faite dans \cite{Lovoronoi1}. Nous en
aurons besoin dans la section \ref{sec_graphes_associes_voronoi}
o\`u nous construisons des graphes associ\'es \`a ce pavage.

Consid\'erons deux surfaces $S$ et $S'$. Soient $p\in S$ et $q\in
S'$. Nous disons que $q$ est \emph{adh\'erent} \`a $p$, not\'e
$q\rightarrow p$, s'il existe un morphisme birationnel $\pi:
S'\rightarrow S$ tel que $p$ est un point \'eclat\'e par $\pi^{-1}$
et $q$ appartient \`a la transform\'ee stricte du diviseur
exceptionnel $E_p$ obtenu en \'eclatant $p$.

Pour construire le pavage de Vorono\"{\i}, nous nous sommes dans un
premier temps restreint \`a un sous-espace, not\'e $\orbl$, de $\H$
contenant l'enveloppe convexe de l'orbite de $\l$ sous l'action du
groupe de Cremona.
    \begin{defi}[{\cite[D\'efinition 2.1]{Lovoronoi1}}]
\label{proprietes_c} {\rm        L'ensemble $\orbl$ est le
sous-espace de $\H$ constitu\'e des classes
\[c=n\ell-\sum\limits_{p\in\B(\P^2)}\lambda_pe_{p} \ \ \text{
(}n\text{ r\'eel}\geq 1\text{)}\] satisfaisant :
        \begin{enumerate}
            \item[\rm 1)] \phantomsection \label{propriete_coeff_classe} $\lambda_p\geq 0$ pour tout $p\in \B(\P^2)$,
            \item[\rm 2)] \phantomsection \label{propriete_classe_canonique} la \emph{positivit\'e contre la classe anti-canonique} : \[ 3n-\sum\limits_{p\in \B(\P^2)}\lambda_p\geq 0,\]
            \item[\rm 3)] \phantomsection \label{propriete_classe_exces} la \emph{positivit\'e des exc\`es} de tout point $p\in\B(\P^2)$ : \[\lambda_{p} - \sum\limits_{\substack{q\in \B(\P^2)\\ q \to p }}\lambda_{q}\geq0,\]
            \item[\rm 4)] \phantomsection \label{propriete_classe_bezout} la \emph{condition de B\'ezout} : pour toute courbe de $\P^2$ de degr\'e $d$ passant avec multiplicit\'e $\mu_p$ en chaque point $p\in\B(\P^2)$ : \[nd-\sum\limits_{p\in \B(\P^2)}\lambda_p\mu_p\geq0.\]
        \end{enumerate}}
    \end{defi}

Le groupe de Cremona agit sur $\orbl$. Nous avons montr\'e que les
cellules de Vorono\"{\i} recouvrent l'espace $\orbl$
(\cite[Corollaire 2.11]{Lovoronoi1}) et forment ainsi un pavage.
Nous avons ensuite \'etudi\'e la g\'eom\'etrie de ce pavage en
caract\'erisant les germes des cellules adjacentes et
quasi-adjacentes (Th\'eor\`emes \ref{cor_cellule_adjacente} et
\ref{cor_cellules_quasi_adjacentes}).

 \subsection{Longueur d'une application du groupe de Cremona}
 J. Blanc et J-P. Furter construisent dans \cite{BlF} un algorithme pour d\'ecomposer une application $f$ du groupe de Cremona en un nombre minimal d'applications de caract\'eristique Jonqui\`eres (ou de \og Jonqui\`eres g\'en\'eralis\'ees \fg{} dans leur terminologie). Ce nombre est appel\'e \og la longueur \fg{} de l'application $f$ et est not\'e $\lgueur(f)$. Ils appellent \og pr\'ed\'ecesseur \fg{} de $f$ une application de caract\'eristique Jonqui\`eres $\j_1$ telle que $f\circ \j_1$ est de degr\'e minimal : pour toute application de caract\'eristique Jonqui\`eres $\j$,
 \[\deg(f\circ\j_1)\leq \deg(f\circ\j ).\]
 Ils montrent que composer une application avec un pr\'ed\'ecesseur fait strictement diminuer le degr\'e et que la longueur d'une application peut \^etre r\'ealis\'ee \`a l'aide de pr\'ed\'ecesseurs.
 \begin{thm}[{\cite[Theorem 1 and Lemma 1.3]{BlF}}]\label{thm_longueur_application}
    Pour tout $f\in\Bir(\P^2)$, il existe une d\'ecomposition minimale en applications de caract\'eristique Jonqui\`eres:
    \[f=\j_n\circ\cdot\circ\j_1\] telle que pour tout $1\leq i\leq n$ les points-base de $\j_i$ sont inclus dans les points-base de $f\circ\j_{1}^{-1} \circ\cdots\circ\j_{i-1}^{-1}$.
    De plus, pour tout $1\leq i\leq n$, \[\deg(f\circ \j_1^{-1}\circ \dots\circ \j_i^{-1})<\deg(f\circ \j_1^{-1}\circ \dots \circ \j_{i-1}^{-1}),\] o\`u $\j_0=\id$.
 \end{thm}

 \subsection{Crit\`ere de Gromov-hyperbolicit\'e}

Le th\'eor\`eme suivant est un crit\`ere pratique pour montrer qu'un
graphe est Gromov-hyperbolique, notamment lorsque nous ne
connaissons pas les g\'eod\'esiques du graphe \'etudi\'e (comme
c'est le cas dans la sous-section \ref{subsection_hyp_graphequasi}
).
\begin{thm}[Crit\`ere de Bowditch \cite{Bow}]\label{thm_bow}
    Soient $h\geq 0$ et $\Gamma$ un graphe connexe muni de la distance standard. Supposons que pour tous $x$ et $y$ appartenant \`a l'ensemble des sommets $S(\Gamma)$, il existe un sous-graphe $\Gamma(x,y)$ de $\Gamma$ connexe par arcs v\'erifiant :
    \begin{enumerate}
        \item[1)] \phantomsection \label{critere_bow1}  Pour tout $x$, $y\in S(\Gamma)$, $x$, $y\in\Gamma(x,y)$.
        \item[2)] \phantomsection \label{critere_bow2} Pour tout $x$, $y$, $z\in S(\Gamma)$, $\Gamma(x,y)\subseteq \Ne_h(\Gamma(x,z)\cup\Gamma(y,z))$ o\`u $\Ne_h$ signifie le $h$-voisinage tubulaire.
        \item[3)] \phantomsection \label{critere_bow3} Pour tout $x$, $y\in S(\Gamma)$ tels que $\dist(x,y)\leq 1$, le diam\`etre de $\Gamma(x,y)$ dans $\Gamma$ est au plus $h$.
    \end{enumerate}
    Alors $\Gamma$ est $\delta$-hyperbolique avec $\delta$ d\'ependant de $h$.
\end{thm}

Cette propri\'et\'e est stable par quasi-isom\'etrie.

\paragraph{D\'efinition (\cite[Definition 8.14]{BH}).}
    Soient deux espaces m\'etriques $(X,d_X)$ et $(X',d_{X'})$. Une application $g : X\rightarrow X'$ est un \emph{plongement quasi-isom\'etrique} s'il existe deux constantes $K\geq 1$ et $L\geq 0$ telles que pour tout $x,y\in X$ :
    \[\frac{1}{K}\dist_X(x,y)-L\leq\dist_{X'}(g(x),g(y))\leq K\dist_X(x,y)+L.\]
    Si de plus il existe une constante $C$ telle que tout \'el\'ement de $X'$ appartient au $C$-voisinage de l'image de $g$, $g$ est
    une \emph{quasi-isom\'etrie}. Dans ce cas, les espaces m\'etriques $X$ et $X'$ sont dits \emph{quasi-isom\'etriques}.

\begin{thm}[{voir par exemple \cite[Th\'eor\`eme p.88]{GdH}}]\label{thm_quasi_isom_hyper}
Soient $(X,\dist_X)$ et $(X',\dist_{X'})$ deux espaces m\'etriques
g\'eod\'esiques quasi-isom\'etriques. L'espace $X$ est
Gromov-hyperbolique si et seulement si l'espace $X'$ est
Gromov-hyperbolique.
\end{thm}

\section{Graphe de Wright}\label{section_wright}
    En 1992, D. Wright a introduit dans son article \cite{W} un complexe simplicial simplement connexe de dimension $2$ sur lequel le groupe de Cremona de rang $2$ sur un corps alg\'ebriquement clos agit. Il permet de voir ce groupe comme un produit amalgam\'e de trois de ses sous-groupes le long de leurs intersections. Dans ce chapitre, nous d\'efinissons ce complexe et montrons que le graphe sous-jacent n'est pas Gromov-hyperbolique. Ceci r\'epond par la n\'egative \`a une question pos\'ee par A. Minasyan et D. Osin dans \cite{MO}. Par cons\'equent, de ce point de vue-l\`a ce n'est pas un analogue du complexe des courbes ou de l'arbre de Bass-Serre.

\subsection{D\'efinition}\label{section_def_Wright}
    Une \emph{surface rationnelle marqu\'ee} est un couple $(S,\varphi)$ o\`u $S$ est une surface rationnelle et $\varphi : S \dashrightarrow \P^2$ est une application birationnelle. L'application $\varphi$ est appel\'ee \emph{marquage}. Soit $n\geq 0$, une surface de Hirzebruch $\F_n$ est une fibration au-dessus de $\P^1$ dont les fibres sont isomorphes \`a $\P^1$ et sont d'auto-intersection $0$, telle qu'il existe une section appel\'ee \emph{section exceptionnelle} d'auto-intersection $-n$. Lorsque $n>0$, cette derni\`ere est unique.
    Consid\'erons les surfaces rationnelles marqu\'ees $(S,\phi)$ o\`u $S$ est isomorphe \`a $\P^2$ ou \`a une surface de Hirzebruch $\F_n$, pour $n\geq 0$.
    Nous disons que deux telles surfaces rationnelles marqu\'ees $(S_1,\phi_1)$ et $(S_2,\phi_2)$ sont \'equivalentes si elles satisfont l'une des deux conditions suivantes (voir Figure \ref{figure_paires_equiv_def_Wright}):
    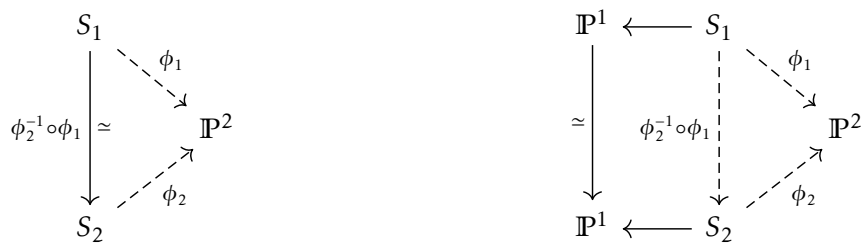
\begin{figure}[h]
        \centering
       \begin{tikzcd}
       \displaystyle S_1 \arrow[dd, rightarrow, "\simeq", "\phi_2^{-1}\circ\phi_1"'] \arrow[dr, dashrightarrow, "\phi_1"] & &\hspace{2cm}& \P^1 \arrow[dd, rightarrow, "\simeq"'] \arrow[r, leftarrow] &   S_1 \arrow[dd, dashrightarrow, "\phi_2^{-1}\circ\phi_1"'] \arrow[dr, dashrightarrow, "\phi_1"]  & \\
       &\P^2 &&    & &\P^2 \\
       S_2 \arrow[ur, dashrightarrow, "\phi_2"']  & && \P^1 \arrow[r, leftarrow] & S_2 \arrow[ur, dashrightarrow, "\phi_2"']  &
       \end{tikzcd}
        \caption{Couples \'equivalents \label{figure_paires_equiv_def_Wright}}
    \end{figure}
    \begin{enumerate}
        \item[1)] $\phi_2^{-1}\circ\phi_1$ est un isomorphisme,
        \item[2)] $S_1$ et $S_2$ sont des surfaces de Hirzebruch d'indices $n_1,n_2>0$ et $\phi_2^{-1}\circ\phi_1$ pr\'eserve la fibration.
    \end{enumerate}

    Nous notons $\overline{(S,\phi)}$ une telle classe d'\'equivalence. Ce sont les sommets du complexe de Wright. Il y a trois types de sommets d\'ependant de la surface $S$ :
    \begin{enumerate}
        \item[$\bullet$] $S$ est isomorphe \`a $\P^2$,
        \item[$\bullet$] $S$ est isomorphe \`a $\F_n$, pour $n\geq1$,
        \item[$\bullet$] $S$ est isomorphe \`a $\F_0$.
    \end{enumerate}
    Tout sommet dont la surface $S$ est isomorphe \`a $\F_n$ pour $n\geq 1$ poss\`ede un repr\'esentant ayant $\F_1$ comme surface.
    Remarquons que si deux couples $(\F_1,\phi_1)$ et $(\F_1,\phi_2)$ sont \'equivalents alors $\phi_2^{-1}\circ \phi_1$ est soit un automorphisme, soit v\'erifie : \[\phi_2^{-1}\circ \phi_1=\pi^{-1}\circ\jonq\circ \pi \] o\`u $\pi^{-1}$ est l'\'eclatement d'un point $p\in \P^2$,  $\jonq$ est une application de Jonqui\`eres pr\'eservant le pinceau de droites passant par le point $p$.

    Trois sommets de type respectivement $\P^2$, $\F_n$ avec $n>0$ et $\F_0$ forment un triangle s'il existe des repr\'esentants respectifs $(\P^2,\phi_1)$, $(\F_1,\phi_2)$ et $(\F_0,\phi_3)$ qui v\'erifient les conditions (voir le diagramme commutatif~\ref{figure_triangle_def_Wright}) :
    \begin{enumerate}
        \item[a)] $\phi_1^{-1}\circ\phi_2$ est une contraction,
        \item[b)] $\phi_3^{-1}\circ\phi_2$ est la compos\'ee de l'\'eclatement d'un point hors de la section exceptionnelle par la contraction de la fibre passant par ce point,
        \item[c)] $\phi_3^{-1}\circ\phi_1$ est la compos\'ee de deux \'eclatements de points de $\P^2$ par la contraction de la droite passant par ces deux points.
    \end{enumerate}
    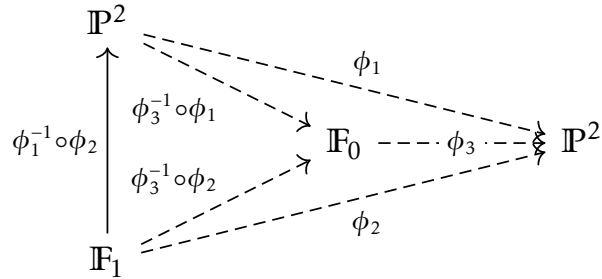
\begin{figure}[h]
        \centering
       \begin{tikzpicture}[baseline= (a).base]
       \node[scale=1.2] (a) at (0,0){
       \begin{tikzcd}
       \P^2  \arrow[drr, dashrightarrow, "\phi_3^{-1}\circ\phi_1"'] \arrow[drrrr, dashrightarrow, "\phi_1"]& & & &\\
       & &\F_0  \arrow[rr, dashrightarrow, "\phi_3" description]& & \P^2 \\
       \F_1 \arrow[urr, dashrightarrow, "\phi_3^{-1}\circ\phi_2"] \arrow[urrrr, dashrightarrow, "\phi_2"'] \arrow[uu, rightarrow,"\phi_1^{-1}\circ\phi_2"]& & & &
       \end{tikzcd} };
       \end{tikzpicture}
        \caption{Sommets r\'ealisant un triangle. \label{figure_triangle_def_Wright}}
    \end{figure}

    Ce complexe de dimension $2$ est simplement connexe (\cite[Theorem 5.5]{W}) mais il n'est pas contractile puisqu'il contient des sph\`eres de dimension $2$ comme par exemple celle de la figure \ref{figure_sphere} o\`u $\sigma$ est l'involution standard de Cremona, $\pi_p$ est l'\'eclatement du point $p\in\P^2$ et $\tau_{p,q}^{-1}$ \'eclate les points $p\in\P^2$ et $q\in\P^2$ puis contracte la droite passant par ces deux points.

    \begin{figure}
        \centering
        \scalebox{0.5}{
        \LARGE
       \begin{tikzpicture}[node distance=3cm, very thick]
       \node (F0 12)   [typetwo,
       label=above left:{$\overline{(\F_0,\tau_{[1:0:0],[0:1:0]})}\quad $}] {};
       \node (F1 2)   [typetwo,
       below left=of F0 12,
       label=left:{$\quad \overline{(\F_1,\pi_{[0:1:0]})}$}] {};
       \node (F0 23)   [typetwo,
       below right=of F1 2,
       label=below left:{$\overline{(\F_0,\tau_{[0:1:0],[0:0:1]})}\quad $}] {};
       \node (F1 1)   [typetwo,
       right=of F0 12,
       label=above right:{$\quad \overline{(\F_1,\pi_{[1:0:0]})}$}] {};
       \node (F0 13)   [typetwo,
       below right=of F1 1,
       label=right:{$\overline{(\F_0,\tau_{[1:0:0],[0:0:1]})}$}] {};
       \node (F1 3)   [typetwo,
       below left=of F0 13,
       label=below right:{$\quad \overline{(\F_1,\pi_{[0:0:1]})}$}] {};
       \node (P2)      [typetwo,
       above right=of F0 12, xshift=-25pt, yshift=50pt,
       label=above:{$\overline{(\P^2,\id)}$}] {};
       \node (P2 2)   [typetwo,
       below left=of F1 3, xshift=25pt, yshift=-50pt,
       label=below:{$\overline{(\P^2,\sigma)}$}] {};

       \draw[dotted] (P2)-- (F1 1) -- (P2 2);
       \draw[dotted] (P2)-- (F0 12) -- (P2 2);
       \draw (P2) -- (F1 2) -- (P2 2);
       \draw (P2) -- (F0 23) -- (P2 2);
       \draw (P2) -- (F1 3) -- (P2 2);
       \draw (P2) -- (F0 13) -- (P2 2);
       \draw[dotted] (F0 13) -- (F1 1) -- (F0 12) -- (F1 2);
       \draw (F1 2) -- (F0 23) -- (F1 3)-- (F0 13);
       \end{tikzpicture} }
        \caption{Une sph\`ere de dimension $2$ dans le complexe.\label{figure_sphere}}
    \end{figure}
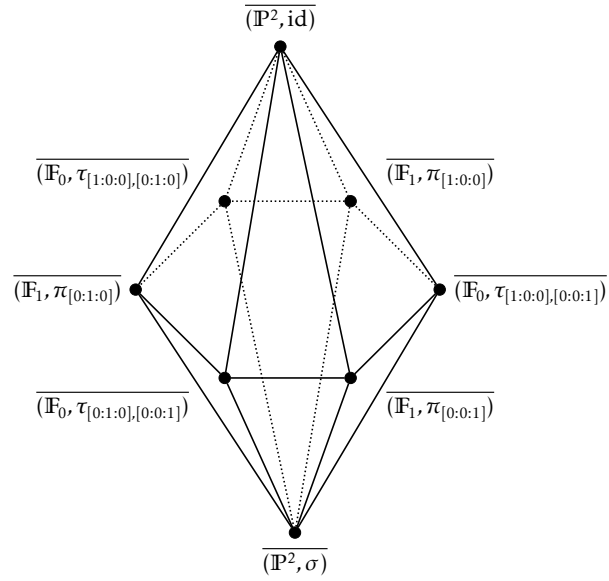

    Le groupe de Cremona agit sur les sommets du complexe de Wright par post-composition : pour toute application $f\in\Bir(\P^2)$ et pour tous repr\'esentant $(S,\phi)$ d'un sommet
    \[f\cdot (S,\phi)= (S,f\circ\phi).\]
    Remarquons que l'action de $f$ pr\'eserve la relation d'\'equivalence sur les sommets et ainsi l'action sur les sommets est bien d\'efinie. De plus, le groupe de Cremona agit par isom\'etries et conserve la structure du complexe.
    Le domaine fondamental de cette action sur le complexe de Wright est un triangle.
    Le graphe de Wright est le $1$-squelette de ce complexe. Il est muni de la m\'etrique standard o\`u les ar\^etes sont isom\'etriques au segment r\'eel $[0,1]$. Nous le notons $\W$.

\subsection{Graphe de Wright modifi\'e}\label{section_graphe_de_Wright_modifie}
    Consid\'erons le graphe $\G$ d\'efini comme suit et appel\'e \emph{graphe de Wright modifi\'e}. Ses sommets sont les sommets de type $\P^2$ du graphe de Wright et nous relions deux sommets s'ils \'etaient \`a distance deux dans le graphe de Wright. Comme nous n'avons plus qu'un type de sommet, nous pouvons oublier la surface et ne consid\'erer que le marquage. Ainsi, les sommets de ce graphe correspondent aux applications du groupe de Cremona munies de la relation d'\'equivalence d\'efinie pour le complexe de Wright, c'est-\`a-dire : \[f\sim g\Leftrightarrow \text{ il existe } a\in\PGL(3,\kk) \text{ tel que } f=g\circ a.\] Une telle classe est not\'ee $\overline{f}$.

    Dans le graphe de Wright, si deux sommets de type $\P^2$ sont \`a distance deux parce qu'il existe un sommet de type $\F_0$ \`a distance un de chacun d'eux, alors il existe \'egalement deux sommets distincts de type $\F_n$ poss\'edant cette propri\'et\'e. Il est donc suffisant de s'int\'eresser aux applications stabilisant les sommets de type $\F_n$.
    Le stabilisateur d'un sommet de type $\F_n$ dans le graphe de Wright est un conjugu\'e du groupe de Jonqui\`eres. Ainsi, dans le graphe de Wright modifi\'e il existe une ar\^ete entre deux sommets $\overline{f}$ et $\overline{g}$ si et seulement s'il existe deux repr\'esentants $f$ et $g$ de $\overline{f}$ et $\overline{g}$ tels que l'un soit obtenu en pr\'e-composant l'autre par un \'el\'ement $\j$ appartenant \`a la classe d'une application de Jonqui\`eres : \[g=f\circ \j .\]
    Montrons que cette d\'efinition est ind\'ependante du choix des repr\'esentants. En effet, consid\'erons deux autres repr\'esentants $f\circ a$ et $g\circ b$ o\`u $a,b\in\PGL(3,\kk)$ de $\overline{f}$ et $\overline{g}$. Nous avons alors :
    \[g\circ b= (f\circ a) \circ (a^{-1}\circ \j \circ a )\circ (a^{-1}\circ b). \]
    Comme $a^{-1}\circ b$ est un automorphisme et que $a^{-1}\circ \j\circ a$ est une application de Jonqui\`eres, la compos\'ee $(a^{-1}\circ \j \circ a )\circ (a^{-1}\circ b)$ appartient \`a la classe d'une application de Jonqui\`eres comme annonc\'e.

    \begin{rmq}\label{rmq_longueur}
{\rm    La distance entre deux sommets $\bar{f}$ et $\bar{g}$ est
donn\'ee par la longueur de l'application $f^{-1}\circ g$ introduit
par J. Blanc et J-P. Furter dans \cite{BlF}.}
    \end{rmq}

    Remarquons que le graphe de Wright modifi\'e est quasi-isom\'etrique au graphe de Cayley du groupe de Cremona pour la famille de g\'en\'erateurs constitu\'ee des transformations de Jonqui\`eres et des automorphismes de $\P^2$.

La proposition suivante justifie le fait que nous nous concentrons
sur le graphe de Wright modifi\'e plut\^ot que sur le graphe de
Wright.

    \begin{prop}\label{prop_quasiisom}
        L'inclusion du graphe de Wright modifi\'e dans le graphe de Wright donn\'ee par l'application $\id : \overline{f} \mapsto \overline{(\P^2,f)}$ est une quasi-isom\'etrie.
    \end{prop}
    \begin{proof}
        Notons $\mathcal{S}(\G)$ et $\mathcal{S}(\W)$ les ensembles des sommets des graphes $\G$ et $\W$.
        Remarquons que dans le cas des graphes, il suffit de montrer qu'il existe une quasi-isom\'etrie entre les sommets. En effet, tout point d'une ar\^ete est \`a distance au plus $\frac{1}{2}$ d'un sommet.
        Consid\'erons l'application : \[\begin{array}{cccc}
        \id:& \mathcal{S}(\G)&\rightarrow& \mathcal{S}(\W)\\
        &\overline{f} &\mapsto &\overline{(\P^2,f)}.
        \end{array}\]
        Par d\'efinition des sommets des deux graphes, elle est bien d\'efinie. Montrons que c'est un plongement quasi-isom\'etrique.

        Remarquons, que dans le graphe de Wright, toute g\'eod\'esique reliant deux sommets diff\'erents de type $\P^2$ passe alternativement par des sommets de type $\P^2$ et des sommets de type $\F_0$ ou $\F_1$. Donc la distance entre deux sommets de type $\P^2$ est paire dans le graphe de Wright. De plus, par d\'efinition du graphe de Wright modifi\'e, tout sommet de type $\P^2$ \`a distance $2$ dans le graphe de Wright correspond \`a un sommet \`a distance $1$ dans le graphe de Wright modifi\'e. Ainsi, pour tous $f,g\in\Bir(\P^2)$, nous avons :
        \[ \dist_{\W}(\overline{(\P^2,f)},\overline{(\P^2,g)})= 2\dist_{\G}(\overline{f},\overline{g}).\]
        Ceci implique que \[\frac{1}{2}\dist_{\G}(\overline{f},\overline{g})\leq\dist_{\W}(\overline{(\P^2,f)},\overline{(\P^2,g)}) \] et que l'application $\id$ est un plongement quasi-isom\'etrique.
        De plus, par construction tout sommet dans le graphe de Wright est \`a distance au plus $1$ d'un sommet de type $\P^2$ donc tout sommet du graphe de Wright est \`a distance au plus $1$ de l'image de $\id$. Ceci ach\`eve la preuve.
    \end{proof}

\subsection{Diam\`etre infini}
    Une premi\`ere question est de savoir si le diam\`etre du graphe modifi\'e est infini ou non. S'il \'etait de diam\`etre fini alors il serait trivialement $\delta$-hyperbolique avec $\delta$ \'egal au diam\`etre.
    Dans cette section, nous montrons que le graphe de Wright modifi\'e est de diam\`etre infini.
    Pour toute application $f\in\Bir(\P^2)$, nous notons $\md(f)$ le nombre de multiplicit\'es distinctes des points-base de $f$.
    \begin{lemme}\label{lemme_mult_pt_base_comp_jonq}
    Soient $g\in\Bir(\P^2)$ et $\j$ une application de caract\'eristique Jonqui\`eres. Le nombre de multiplicit\'es distinctes des points-base de $g\circ\j$ est inf\'erieur ou \'egal \`a \[\md(g\circ\j)\leq 2\md(g)+2.\]
    \end{lemme}
        \begin{proof}
        Les points-base de $g\circ\j$ sont inclus dans l'ensemble des points suivants :
                \begin{enumerate}
                    \item[1)] \phantomsection \label{enum_pt_base_du_dessus} l'image par $\j^{-1}$ des points-base de $g$ qui ne sont pas des points-base de $\j^{-1}$,
                    \item[2)] \phantomsection \label{enum_pt_base_du_dessous} les points-base de $\j$.
                \end{enumerate}
                Dans le cas \hyperref[enum_pt_base_du_dessus]{1)}, appliquer $\j^{-1}$ ne modifie pas le nombre de multiplicit\'es distinctes de ces points. En effet, la multiplicit\'e de chacun de ces points-l\`a est la m\^eme pour l'application $\j$ ou pour l'application $g\circ\j$. Ce n'est pas le cas des points de \hyperref[enum_pt_base_du_dessous]{2)} qui ne poss\`edent pas la m\^eme multiplicit\'e en tant que points-base de $\j$ ou de $g\circ\j$. Notons $m_p(f)$ la multiplicit\'e du point $p$ en tant que point-base de $f\in\Bir(\P^2)$ avec pour convention qu'elle est \'egale \`a $0$ si $p$ n'est pas un point-base de $f$. Notons $p_0$ le point-base de multiplicit\'e maximale de $\j$ qui est \'egalement celui de $\j^{-1}$, et $p_1,\dots,p_{2d-2}$ et $q_1,\dots,q_{2d-2}$ les petits points-base respectifs de $\j$ et $\j^{-1}$. Alors, en utilisant la formule (\ref{eq_action_f_sur_c}) (l'application $f$ de la formule est remplac\'ee par $\j^{-1}$ et $c$ par $\isome{g}^{-1}(\l)$), nous obtenons pour tout $1\leq i\leq 2d-2$ :   \begin{equation}\label{eq_mult}
                m_{p_i}(g\circ\j)=\deg(g)-m_{p_0}(g)-m_{q_i}(g).
                \end{equation}
                C'est la position des points-base de $\j^{-1}$ par rapport aux points-base de $g$ qui va d\'eterminer le nombre de multiplicit\'es diff\'erentes pour les points-base de $g\circ\j$.
                Nous allons avoir au maximum :
                \begin{enumerate}[leftmargin=0pt, labelwidth=-24pt, listparindent=\parindent]
                    \item[$\bullet$] une multiplicit\'e qui correspond au point-base maximal de $\j$,
                    \item[$\bullet$] une autre qui correspond au fait que certains petits points-base de $\j^{-1}$ ne sont pas des points-base de $g$ (d'apr\`es la formule \eqref{eq_mult}, cette multiplicit\'e est \'egale au degr\'e de $g$ si $p_0$ n'est pas un point-base de $g$ et \`a $\deg(g)-m_{p_0}(g)$ sinon),
                    \item[$\bullet$] $\md(g)$ multiplicit\'es diff\'erentes qui correspondent aux petits points-base de $\j^{-1}$ qui sont aussi des points-base de $g$ dont les multiplicit\'es sont donn\'ees par la formule \eqref{eq_mult},
                    \item[$\bullet$] $\md(g)$ multiplicit\'es diff\'erentes qui correspondent aux points-base de $g\circ\j$ qui sont images par $\j^{-1}$ des points-base de $g$ (qui ne sont pas des points-base de $\j$). Remarquons que pour ces points-l\`a leur multiplicit\'e pour $g$ ou pour $g\circ\j$ est identique.
                \end{enumerate}
                Ainsi, l'ensemble des points-base de l'application $g\circ\j$ poss\`ede au plus $2\md(g)+2$ multiplicit\'es diff\'eren\-tes.~\end{proof}
    \begin{lemme}\label{lemme_ptbase_compose_jonquieres}
        Soit $\{\j_n\}_{n\in\N^*}$ une suite d'applications de Jonqui\`eres. Pour tout $n\in\N^*$ le nombre de multiplicit\'es distinctes des points-base de $\j_1\circ\j_2\circ\dots\circ \j_n$ est inf\'erieur ou \'egal \`a $2^{n+1}-2$.
    \end{lemme}

    \begin{proof}
        Raisonnons par r\'ecurrence. Une application de Jonqui\`eres ne poss\`ede que deux multiplicit\'es distinctes donc le cas $n=1$ est v\'erifi\'e. Supposons que le r\'esultat soit vrai pour la compos\'ee de $n$ applications de Jonqui\`eres, montrons qu'il reste vrai pour $\j_{1}\circ\dots\circ \j_{n+1}$, la compos\'ee de $n+1$ applications de Jonqui\`eres. Notons $g=\j_1\circ\dots \circ\j_n$.
        Par hypoth\`ese de r\'ecurrence et par le lemme \ref{lemme_mult_pt_base_comp_jonq}, le nombre de multiplicit\'es diff\'erentes est  au maximum $2+2(2^{n+1}-2)=2^{n+2}-2$  comme annonc\'e.
    \end{proof}

    \begin{prop}\label{prop_diam_infini}
        Le graphe de Wright modifi\'e est de diam\`etre infini.
    \end{prop}

    \begin{proof}
        Nous allons montrer par l'absurde que le graphe de Wright modifi\'e est de diam\`etre infini.
        Supposons que le diam\`etre $D$ de $\G$ soit fini. Choisissons un entier $n$ tel que $n>2^{D+1}-2$. Consid\'erons une suite d'applications quadratiques $\{\q_i\}_{1\leq i\leq n}$ telle que pour tout $i\in\{2,\dots,n\}$, les points-base de $\q_i^{-1}$ sont disjoints des points-base de $\q_{1}\circ\dots\circ\q_{i-1}$. La caract\'eristique de cette application est $(2^n;m_n^3,\dots,m_1^3)$ o\`u $m_i=2^i$. Cela d\'ecoule du fait que si les points-base de $f^{-1}$ sont disjoints des points-base de $g$ alors les points-base de $g\circ f$ sont les points-base de $f$ dont la multiplicit\'e a \'et\'e multipli\'ee par $\deg(g)$ et l'image par $f^{-1}$ des points-base de $g$ qui conservent la m\^eme multiplicit\'e. De plus, le degr\'e de $g\circ f$ est \'egal au produit des degr\'es de $f$ et de $g$. En particulier, l'application $\q_1\circ\dots \circ\q_n$ poss\`ede $n$ multiplicit\'es distinctes. D'apr\`es le lemme \ref{lemme_ptbase_compose_jonquieres}, cela signifie que le nombre minimal d'applications de Jonqui\`eres permettant de d\'ecomposer $\q_1\circ\dots \circ\q_n$ est strictement sup\'erieur \`a $D$. Ceci implique que $\dist_{\G}(\overline{\id},\overline{\q_1\circ\dots\circ\q_n})>D$ ce qui contredit le fait que le diam\`etre est born\'e.
    \end{proof}

    Le corollaire suivant est d\^u \`a la quasi-isom\'etrie du graphe de Wright et du graphe de Wright modifi\'e.

    \begin{cor}\label{cor_Wright_infini}
        Le graphe de Wright est de diam\`etre infini.
    \end{cor}

\subsection{Non-hyperbolicit\'e}
    Nous montrons dans cette section que le graphe de Wright modifi\'e n'est pas Gromov-hyperbolique. Pour cela, nous montrons qu'il existe un plongement quasi-isom\'etrique du graphe de Cayley de $\Z^2$ dans le graphe de Wright modifi\'e. Nous utilisons un sous-groupe du groupe de Cremona isomorphe \`a $\Z^2$, constitu\'e de deux applications particuli\`eres de la famille des twists d'Halphen.
    \subsubsection{Construction d'un sous-groupe ab\'elien libre de rang huit}\label{subsubsection_groupe_abelien} Les surfaces de Halphen sont construites de la fa\c{c}on suivante (voir \cite[p.24]{Ca} ou encore \cite[Section 3.4.1]{Lothese}).
    Soit $C_0$ une cubique lisse de $\P^2$. Fixons la loi de groupe sur $C_0$ de sorte que le neutre soit un point d'inflexion. Choisissons $9$ points sur cette courbe de sorte que la somme de ces points soit \'egale au neutre. Ainsi, nous avons un pinceau de courbes elliptiques passant par ces $9$ points. En les \'eclatant, nous obtenons une fibration sur $\P^1$ et dont les fibres sont les transform\'ees strictes de ce pinceau. La surface obtenue est appel\'ee \og surface de Halphen \fg{}.

     Soit $X$ une surface de Halphen obtenue en \'eclatant les points $p_0, \dots, p_8$ de $\P^2$. Le groupe de N\'eron-Severi de $X$ est engendr\'e dans l'espace de Picard-Manin par :
         \[\Pic(X)=\langle\l,e_{p_0},\dots,e_{p_8}\rangle.\] La classe du diviseur canonique de la surface $X$ est \'egale \`a \[\can_X=-3\l+e_{p_0}+\dots + e_{p_8}.\]
         De plus, \[\can_X^{\perp}/\langle\can_X\rangle=\langle e_{p_1}-e_{p_0},e_{p_2}-e_{p_0}\dots ,e_{p_8}-e_{p_0}\rangle\simeq\Z^8.\]
         D'apr\`es \cite[Theorem 2.10]{CaDol}, le groupe des automorphismes de $X$ contient un sous-groupe isomorphe \`a $\Z^8$. Notons $\{f_i\}_{1\leq i\leq 8}$ un syst\`eme g\'en\'erateur de ce sous-groupe.
        Soit $\pi :X\rightarrow\P^2$ l'\'eclatement des points $p_0,\dots, p_8$, les applications du groupe de Cremona $\{\pi\circ f_i\circ \pi^{-1}\}_{1\leq i\leq 8}$ engendrent un sous-groupe du groupe de Cremona isomorphe \`a $\Z^8$. Remarquons que pour $1\leq i\leq 8$, les applications $\pi\circ f_i\circ \pi^{-1}$ not\'ees $g_i$ poss\`edent au plus $9$ points-base. C'est \'egalement le cas de toutes les applications de ce $\Z^8$.

 \subsubsection{Non-hyperbolicit\'e du graphe de Wright}
   Dans cette sous-section, nous consid\'erons deux g\'en\'erateurs du sous-groupe ab\'elien de rang $8$ construit pr\'ec\'edemment. Nous montrons que le graphe de Cayley du sous-groupe $\Z^2$ obtenu se plonge quasi-isom\'etriquement dans le graphe de Wright modifi\'e.

    Nous avons besoin de la proposition suivante tir\'ee de \cite[(11) p.874]{CaDol}.
    \begin{prop}\label{prop_formule_action_halphen}
        \`A chaque \'el\'ement de $a\in\can_X^{\perp}/\langle\can_X\rangle$ est associ\'e un automorphisme de $X$, et donc une isom\'etrie $\tau_a$ de $\Pic(X)$ via la formule :
        \[\begin{array}{cccc}
        \tau_a : &\Pic(X) &\rightarrow &\Pic(X)\\
        &d&\mapsto & d-(\can_X\cdot\ d)a+(a\cdot d-\frac{1}{2}(\can_X\cdot\   d)(a\cdot a))\can_X
        \end{array}. \]
    \end{prop}

     En fait, les applications $f_i$ du paragraphe pr\'ec\'edent peuvent \^etre choisies de fa\c{c}on \`a correspondre \`a la translation $\tau_{e_{p_i}-e_{p_0}}$.

    \begin{cor}\label{cor_degre_halphen}
        Soient $g_1$ et $g_2$ les deux applications construites pr\'ec\'edemment. Alors pour tous $n,m\in\Z$ nous avons : \[\deg(g_1^n\circ g_2^m)=9(n^2+m^2+mn)+1.\]
    \end{cor}
    \begin{proof}
        Les deux applications $g_1$ et $g_2$ se rel\`event en des automorphismes de $X$ correspondant respectivement aux translations $\tau_{e_{p_1}-e_{p_0}}$ et $\tau_{e_{p_2}-e_{p_0}}$. L'it\'er\'ee de $g_1$ se rel\`eve en l'it\'er\'ee de la translation correspondante. De plus, translater $n$ fois par $e_{p_1}-e_{p_0}$ revient \`a translater une fois par $n(e_{p_1}-e_{p_0})$, ainsi : \[\tau_{e_{p_1}-e_{p_0}}^n=\tau_{n(e_{p_1}-e_{p_0})}. \] De m\^eme, pour tous $n,m\in\N$, la translation correspondant \`a $g_1^n\circ g_2^m$ est \'egale \`a \[\tau_{e_{p_1}-e_{p_0}}^n\circ\tau_{e_{p_2}-e_{p_0}}^m=\tau_{ne_{p_1}+me_{p_2}-(m+n)e_{p_0}}. \]
        De plus, la surface $X$ domine la surface de r\'esolution minimale de $g_1^n\circ g_2^m$ ainsi les actions de $g_1^n\circ g_2^m$ et de $\tau_{e_{p_1}-e_{p_0}}^n\circ\tau_{e_{p_2}-e_{p_0}}^m$ sur l'espace de Picard-Manin co\"{\i}ncident. Par cons\'equent, pour tous $n,m\in\N$,
        \[\deg(g_1^n\circ g_2^m)=\tau_{ne_{p_1}+me_{p_2}-(m+n)e_{p_0}}(\l)\cdot\l.\]
        Par la proposition \ref{prop_formule_action_halphen}, nous avons :
        \[\tau_{ne_{p_1}+me_{p_2}-(m+n)e_{p_0}}(\l)= \l +3(ne_{p_1}+me_{p_2}-(n+m)e_{p_0})-3(n^2+m^2+nm)\can_X.\]
        Ainsi nous obtenons que le degr\'e de $g_1^n\circ g_2^m$ est \'egal \`a
        \[\deg(g_1^n\circ g_2^m)=9(n^2+m^2+nm)+1,\] comme annonc\'e.
    \end{proof}

    Le lemme suivant est une adaptation du lemme de J. Blanc et S. Cantat (\cite[Lemma 5.10]{BC}). Je remercie J. Blanc de m'avoir fait remarquer qu'il s'adaptait dans ce cas-l\`a. Dans leur preuve, il suffit de consid\'erer le diviseur canonique associ\'e \`a la surface obtenue en \'eclatant $9$ points contenant l'ensemble $\Bs(f_2)\cup\Bs(f_1^{-1})$.
    \begin{lemme}\label{lemme_quadratique}
    Si $f_1$ et $f_2$ sont deux applications du groupe de Cremona telles que le cardinal de $\Bs(f_2)\cup\Bs(f_1^{-1})$ est inf\'erieur ou \'egal \`a $9$ alors elles satisfont l'in\'egalit\'e : \[\sqrt{\deg(f_2\circ f_1)}\leq \sqrt{\deg(f_2)}+\sqrt{\deg(f_1)}.\]
    \end{lemme}

   Le lemme suivant est extrait de \cite[Lemme 4.22]{BlF}.
    \begin{lemme}\label{lemme_deg_fonction_longueur}
    Soit $h\in\Bir(\P^2)$ telle que $h$ poss\`ede au plus $9$ points-base. Nous avons alors :
    \[\deg(h)\leq 5\lgueur(h)^2.\]
    \end{lemme}
    \begin{proof}
    Montrons ce r\'esultat par r\'ecurrence sur la longueur de l'application.
    Supposons que $h$ soit de longueur $1$, ce qui signifie que $h$ est une application de caract\'eristique Jonqui\`eres qui poss\`ede au plus $9$ points-base. Par cons\'equent $h$ est de degr\'e inf\'erieur ou \'egal \`a $5$.
    Supposons le r\'esultat vrai pour toute application de longueur inf\'erieure ou \'egale \`a $k-1$ et poss\'edant au plus $9$ points-base. Montrons le r\'esultat pour une application $h$ de longueur $k$ et ayant au plus $9$ points-base.
    D'apr\`es le th\'eor\`eme \ref{thm_longueur_application}, nous pouvons d\'ecomposer $h$ de fa\c{c}on minimale $h=\j_k\circ\cdots\circ\j_1$ \`a l'aide d'applications de caract\'eristique Jonqui\`eres $\j_1,\dots,\j_k$ de sorte que les points-base de $\j_1$ soient inclus dans les points-base de $h$. Notons $h_2= a\circ \j_k\circ\cdots\circ\j_2$. Cette application est de longueur $k-1$ et se r\'e\'ecrit : $h_2=h\circ j_1^{-1}$. L'ensemble des points-base de $h_2$ est inclus dans l'ensemble $E$ de points suivant :
    \begin{enumerate}
    \item[$\bullet$] l'ensemble des points-base de $j_1^{-1}$,
    \item[$\bullet$] l'image par $j_1$ des points-base de $h$ qui ne sont pas des points-base de $j_1$.
    \end{enumerate} L'ensemble des points-base de $j_1$ \'etant inclus dans celui des points-base de $h$, le cardinal de $E$ est major\'e par :
    \[\card(E)\leq\card(\Bs( j_1^{-1}))+ \card(\{p\in\Bs(h)\mid p\notin \Bs(j_1)\})=\card(h)\leq 9.\] Par cons\'equent le cardinal de l'ensemble $\Bs(h_2)\cup\Bs(j_1^{-1})$ est inf\'erieur ou \'egal \`a $9$. D'apr\`es le lemme \ref{lemme_quadratique}, le degr\'e de la compos\'ee satisfait :
    \[\sqrt{\deg(h_2\circ j_1)}\leq \sqrt{\deg(h_2)}+\sqrt{\deg(j_1)}.\]
    Comme les applications $j_1$ et $h_2$ ont toutes deux au plus $9$ points-base, nous concluons en utilisant le fait que le degr\'e de $j_1$ est inf\'erieur ou \'egal \`a $5$ et par hypoth\`ese de r\'ecurrence.
    \end{proof}

   \begin{thm}\label{thm_qi_gc_z2}
   Le graphe de Wright modifi\'e poss\`ede un sous-graphe quasi-isom\'etrique au graphe de Cayley de $\Z^2$, not\'e $\GC(\Z^2)$. Plus pr\'ecis\'ement,
   en consid\'erant les applications $g_1,g_2\in\Bir(\P^2)$ construites pr\'ec\'edemment, l'inclusion :
    \[\begin{array}{cccc}
        \mathcal{I}:& \GC(<g_1,g_2>)&\rightarrow& \G \\
        & f &\mapsto &\overline{f}
        \end{array}\] est une quasi-isom\'etrie sur son image.
   \end{thm}
   \begin{proof}
   Un \'el\'ement du groupe $<g_1,g_2>$ s'\'ecrit $g_1^m\circ g_2^n$ o\`u $m,n\in\Z$. Prouvons que c'est une quasi-isom\'etrie plong\'ee. Par transitivit\'e du groupe $<g_1,g_2>$, nous pouvons supposer qu'un des sommets est $\id$.
   Pour tous $m,n\in\Z$, la distance $\dist_{\G}\big((g^m_1\circ g_2^n)(\overline{\id}),\overline{\id}\big)$ correspond \`a la longueur de l'application $g^m_1\circ g_2^n$. Ces applications ont au plus $9$ points-base, par cons\'equent nous pouvons leur appliquer le lemme \ref{lemme_deg_fonction_longueur} :
      \begin{equation}\label{eq_degre_fonction_longueur}
      \deg(g^m_1\circ g_2^n)\leq 5\lgueur(g^m_1\circ g_2^n)^2.
      \end{equation}
    D'apr\`es le corollaire \ref{cor_degre_halphen}, pour tout $m,n\in\Z$,
        \[5(n+m)^2\leq9(n^2+m^2+mn)+1= \deg(g^m_1\circ g_2^n) .\]
     Ceci prouve que \[\dist_{\GC(\Z^2)}(g^m_1\circ g_2^n,\id)\leq \dist_{\G}\big((g^m_1\circ g_2^n)(\overline{\id}),\overline{\id}\big).\]
     De plus, en notant $K=\max(\lgueur(g_1),\lgueur(g_2))$ le maximum entre la longueur de $g_1$ et la longueur de $g_2$ nous obtenons :
     \[\frac{1}{K}\dist_{\G}\big((g^m_1\circ g_2^n)(\overline{\id}),\overline{\id}\big)\leq \dist_{\GC(\Z^2)}(g^m_1\circ g_2^n,\id).\]
     Enfin sur son image, l'application $\mathcal{I}$ est une quasi-isom\'etrie.
   \end{proof}
   \begin{cor}\label{cor_wright_modif_pas_hyperbolique}
   Le graphe de Wright modifi\'e n'est pas Gromov-hyperbolique.
   \end{cor}

  Le th\'eor\`eme \ref{main_thm_Wright_nonhyperbolique} d\'ecoule de la proposition \ref{prop_quasiisom} et du corollaire pr\'ec\'edent.

    \section{Graphes associ\'es au pavage de Vorono\"{\i}}\label{sec_graphes_associes_voronoi}
        Le but de cette partie est d'exhiber un graphe Gromov-hyperbolique naturel sur lequel le groupe de Cremona agit. Le premier candidat est \og le graphe d'adjacence \fg{} qui est une sorte de graphe dual au pavage de Vorono\"{\i}, codant la g\'eom\'etrie des cellules adjacentes. Nous montrons qu'il est quasi-isom\'etrique au graphe de Wright modifi\'e. Par cons\'equent, il est \'egalement quasi-isom\'etrique au graphe de Wright. Cela permet de retrouver le graphe de Wright dans l'espace hyperbolique $\H$. Mais cela implique qu'il n'est pas Gromov-hyperbolique puisque nous avons montr\'e que le graphe de Wright modifi\'e ne l'est pas. Nous nous int\'eressons ensuite au graphe \og quasi-adjacent \fg{} li\'e \`a la g\'eom\'etrie des cellules quasi-adjacentes. Nous montrons qu'il reste de diam\`etre infini et qu'il est Gromov-hyperbolique. Il peut ainsi \^etre vu comme un analogue du graphe des courbes.

        \subsection{Graphe d'adjacence}
        Construisons le graphe d'adjacence. Les sommets de ce graphe sont les centres des cellules de Vorono\"{\i}, c'est-\`a-dire l'orbite de $\l$ sous l'action du groupe de Cremona : \[\{\isome{f}(\l) \mid f\in\Bir(\P^2)\}.\] Remarquons que comme le groupe $\PGL(3,\kk)$ est le stabilisateur de la classe $\l$, les sommets du graphe d'adjacence sont en bijection avec les classes \`a gauche du groupe de Cremona modulo $\PGL(3,\kk)$. Nous relions deux sommets par une ar\^ete si les deux cellules de Vorono\"{\i} correspondantes sont adjacentes entre elles. Ce graphe est muni de la m\'etrique standard et nous le notons $\D$.
        D'apr\`es le corollaire \ref{cor_cellule_adjacente}, les germes des cellules adjacentes \`a la cellule $\V(\id)$ sont des applications de caract\'eristiques Jonqui\`eres ou sont des applications dont les points-base de l'inverse sont en position presque g\'en\'erale. Par cons\'equent la distance entre les sommets $\isome{f}(\l)$ et $\l$ du graphe d'adjacence correspond au nombre minimal d'applications de ces deux types qu'il faut pour d\'ecomposer l'application $f$.

        \begin{rmq}{\rm
        Le graphe d'adjacence est quasi-isom\'etrique au graphe de Cayley du groupe de Cremona avec pour famille de g\'en\'erateurs $\PGL(3,\kk)$, les applications de Jonqui\`eres et les applications ayant au plus $8$ points-base en position presque g\'en\'erale.}
        \end{rmq}

        \begin{prop}\label{prop_qi_graphes}
        L'application  \[\begin{array}{cccc}
                                \mathcal{I} :& \mathcal{S}(\D)&\rightarrow& \mathcal{S}(\G)\\
                                & \isome{f}(\l) &\mapsto &\bar{f}
                                \end{array}\] entre les sommets du graphe d'adjacence $\mathcal{S}(\D)$ et les sommets du graphe de Wright modifi\'e $\mathcal{S}(\G)$ est une quasi-isom\'etrie. En particulier, le graphe de Wright $\W$, le graphe de Wright modifi\'e $\G$ et le graphe d'adjacence $\D$ sont quasi-isom\'etriques.
        \end{prop}
        \begin{proof}
        L'application $\mathcal{I}$ est bien d\'efinie car les deux ensembles de sommets sont d\'efinis par les classes \`a gauche du groupe de Cremona modulo $\PGL(3,\kk)$. Prouvons dans un premier temps que c'est un plongement quasi-isom\'etrique. Le groupe de Cremona agissant transitivement sur les sommets des deux graphes, il suffit de montrer les deux in\'egalit\'es sur les couples $(\isome{f}(\l),\isome{\id}(\l))$ o\`u $f\in\Bir(\P^2)$.

            Rappelons que la distance d'un sommet $\bar{f}$ du graphe de Wright modifi\'e au sommet $\bar{\id}$ est donn\'ee par la longueur de l'application $f$ (voir la remarque \ref{rmq_longueur}). Ainsi, nous avons : \[\dist_{\D}(\isome{f}(\l),\isome{\id}(\l))\leq\dist_{\G}(\overline{f},\overline{\id}).\]
            Les applications correspondant aux cellules adjacentes qui ne sont pas des applications de caract\'eristiques Jonqui\`eres sont de degr\'e au plus $17$ (\cite[Corollaire 4.9]{Lovoronoi1}).
            D'apr\`es le th\'eor\`eme \ref{thm_longueur_application}, la longueur minimale d'une application peut \^etre atteinte en composant des applications de caract\'eristique Jonqui\`eres de sorte que le degr\'e augmente strictement \`a chaque \'etape. Ainsi, toute application de degr\'e inf\'erieur ou \'egal \`a $17$ est la compos\'ee d'au plus $16$ applications de caract\'eristique Jonqui\`eres. Par cons\'equent, nous avons :
            \[\dist_{\G}(\overline{f},\overline{\id})\leq 16\dist_{\D}(\isome{f}(\l),\isome{\id}(\l)).\]
            Ceci ach\`eve de montrer que l'application $\mathcal{I}$ est un plongement quasi-isom\'etrique.
            De plus, les sommets des deux graphes sont en bijection donc $\mathcal{I}$ est une quasi-isom\'etrie. Cela implique que les graphes $\D$ et $\G$ sont quasi-isom\'etriques. De plus, nous avons d\'ej\`a montr\'e que les graphes $\W$ et $\G$ sont quasi-isom\'etriques (Proposition \ref{prop_quasiisom}).
        \end{proof}

        Cette proposition implique d'une part que le graphe d'adjacence construit est de diam\`etre infini et d'autre part qu'il n'est pas hyperbolique au sens de Gromov puisque le graphe de Wright modifi\'e ne l'est pas (Th\'eor\`eme \ref{thm_quasi_isom_hyper} et Corollaire \ref{cor_wright_modif_pas_hyperbolique}).

        \subsection{Graphe de quasi-adjacence}
        Les sommets du graphe de quasi-adjacence $\overline{\D}$ sont les centres des cellules de Vorono\"{\i} et nous relions deux sommets par une ar\^ete lorsque les cellules correspondantes sont quasi-adjacentes. Dans cette partie, nous \'etudions la Gromov-hyperbolicit\'e du graphe de quasi-adjacence.

         Par la remarque \ref{rmq_adjacent_sousgraphe_quasi}, le graphe d'adjacence est un sous-graphe du graphe de quasi-adjacence. D'apr\`es les th\'eor\`emes \ref{cor_cellule_adjacente} et \ref{cor_cellules_quasi_adjacentes}, les germes des sommets qui deviennent \`a distance un du sommet $\l$ sont les applications birationnelles qui ne sont pas de caract\'eristique Jonqui\`eres et qui poss\`edent exactement $9$ points-base en position presque g\'en\'erale. Ces applications fixent ou \'echangent les classes au bord de $\H$ de la forme $3\l-\sum_{i=0}^{8}e_{p_i}$ o\`u les points $p_i$ sont en position presque g\'en\'erale (Lemme \cite[Lemme 5.15]{Lovoronoi1}). Ainsi, les \'el\'ements qui nous g\^enaient pour l'hyperbolicit\'e du graphe d'adjacence deviennent elliptiques dans le graphe de quasi-adjacence. Il est par cons\'equent un bon candidat \`a \^etre hyperbolique.
        Afin que cela soit int\'eressant, nous devons montrer qu'il reste de diam\`etre infini.

        \subsubsection{Diam\`etre infini}
        Nous avons vu que le graphe d'adjacence est de diam\`etre infini. Nous montrons que cela reste le cas pour le graphe de quasi-adjacence.

        \begin{prop}\label{prop_diam_infini_adjacent_etendu}
            Le graphe de quasi-adjacence est de diam\`etre infini.
        \end{prop}

        Pour montrer cette proposition nous modifions l\'eg\`erement la preuve faite pour montrer que le graphe de Wright modifi\'e est de diam\`etre infini (Proposition \ref{prop_diam_infini}). Le lemme suivant est l'analogue du lemme \ref{lemme_ptbase_compose_jonquieres}.

        \begin{lemme}\label{lemme_ptbase_compose_applibord_id}
            Soit $\{f_n\}_{n\in\N^*}$ une suite d'applications telle que pour tout $n$, $f_n$ est soit une application de caract\'eristique Jonqui\`eres soit une application ayant au plus $9$ points-base. Pour tout $n\in\N^*$, le nombre de multiplicit\'es diff\'erentes de $f_1\circ f_2\circ \dots\circ f_n$ est inf\'erieur ou \'egal \`a $2^{n+3}-2$.
        \end{lemme}
        \begin{proof}
            Raisonnons par r\'ecurrence.
            Si $n=1$, alors $f$ est soit une application de caract\'eristique Jonqui\`eres soit une application ayant au plus $9$ points-base. Dans le premier cas elle poss\`ede deux multiplicit\'es diff\'erentes et dans le second cas elle en poss\`ede au plus $9$. Dans les deux cas elle en poss\`ede moins de $14$.

            Supposons le r\'esultat vrai pour $f_1\circ f_2\circ\dots\circ f_n$. Montrons qu'il reste vrai pour l'application $f_{1}\circ\dots\circ f_{n+1}$. Notons $g=f_1\circ\dots \circ f_n$.
            Si $f_{n+1}$ est une application de caract\'eristique Jonqui\`eres alors d'apr\`es le lemme \ref{lemme_mult_pt_base_comp_jonq} et par hypoth\`ese de r\'ecurrence, l'application $g\circ f_{n+1}$ poss\`ede au plus \[2+2(2^{n+3}-2)=2^{n+4}-2\] multiplicit\'es diff\'erentes.

            Sinon $f_{n+1}$ poss\`ede au plus $9$ points-base et l'ensemble des points-base de $g\circ f_{n+1}$ est inclus dans l'ensemble de points suivant :
            \begin{enumerate}
                \item[1)] \phantomsection \label{enum_pt_base_du_dessus2} l'image par $f_{n+1}^{-1}$ des points-base de $g$ qui ne sont pas des points-base de $f_{n+1}^{-1}$,
                \item[2)] \phantomsection \label{enum_pt_base_du_dessous2} ainsi que des points-base de $f_{n+1}$.
            \end{enumerate}
            La multiplicit\'e des points du cas \hyperref[enum_pt_base_du_dessus2]{1)} est la m\^eme pour l'application $f_{n+1}$ et pour l'application $g\circ f_{n+1}$.
            Par hypoth\`ese de r\'ecurrence, il y a au maximum $9$ multiplicit\'es distinctes correspondant aux points-base de $f_{n+1}$ et $2^{n+3}-2$ multiplicit\'es diff\'erentes en plus correspondant aux points-base de $g$ qui ne sont pas des points-base de $f_{n+1}^{-1}$. Cela implique que $f_{1}\circ\dots\circ f_{n+1}$ poss\`ede au plus \[2^{n+3}-2+9 =2^{n+3}-7\leq 2^{n+4}-2\]
            multiplicit\'es diff\'erentes. Le r\'esultat est ainsi d\'emontr\'e.
        \end{proof}

        \begin{proof}[Preuve de la Proposition \ref{prop_diam_infini_adjacent_etendu}]
            Nous allons montrer par l'absurde que le graphe de quasi-adjacence est de diam\`etre infini.
            Supposons que le diam\`etre de $\bar{\D}$ soit de diam\`etre $D$ born\'e. Choisissons un entier $n$  strictement sup\'erieur \`a $2^{D+3}-2$ : \[n>2^{D+3}-2.\] Consid\'erons une suite d'applications quadratiques $\{\q_i\}_{1\leq i\leq n}$ telle que pour tout $i\in\{2,\dots,n\}$, les points-base de $\q_i^{-1}$ sont disjoints des points-bases de $\q_{1}\circ\dots\circ\q_{i-1}$. Ainsi l'application $\q_1\circ\dots\circ \q_n$ poss\`ede $n$ multiplicit\'es diff\'erentes (voir la preuve de la proposition \ref{prop_diam_infini}). D'apr\`es le lemme \ref{lemme_ptbase_compose_applibord_id}, cela signifie que cette application se d\'ecompose en strictement plus de $D$ applications qui sont ou de caract\'eristique Jonqui\`eres ou qui poss\`edent au plus $9$ points-base. Ceci implique que $\dist_{\bar{\D}}\big(\isome{\id}(\l),\isome{(\q_1\circ\dots\circ\q_n)}(\l)\big)>D$ ce qui contredit le fait que le diam\`etre est born\'e.
        \end{proof}

        \subsection{Hyperbolicit\'e du graphe de quasi-adjacence}\label{subsection_hyp_graphequasi}

        Pour montrer que le graphe de quasi-adjacence est Gromov-hyperbolique nous utilisons le crit\`ere de Bowditch \ref{thm_bow} et le fait que l'espace $\orbl$ est Gromov-hyperbolique.

        \subsubsection{Construction de sous-graphes}\label{subsubsection_construction_sous_graphes}
        Nous construisons des sous-graphes dont nous montrons qu'ils v\'erifient le crit\`ere de Bowditch.

        Dans un premier temps, nous associons \`a tout \'el\'ement de $\orbl$ un sous-graphe du graphe de quasi-adjacence. Un choix naturel du point de vue du pavage de Vorono\"{\i} est de faire correspondre \`a $c\in\orbl$ tous les sommets du graphe de quasi-adjacence correspondant aux cellules de Vorono\"{\i} auxquelles $c$ appartient :
        \[\Som_c:=\{\isome{f}(\l)\mid c\in\V(f)\}.\]

        Notons $\bar{\D}_c$ le graphe complet associ\'e aux sommets de $\Som_c$. C'est un sous-graphe de $\bar{\D}$. L'application qui \`a $c$ associe $\bar{\D}_c$ est $\Bir(\P^2)$-\'equivariante : pour tout $f\in\Bir(\P^2)$, pour tout $c\in\orbl$ :\[f\cdot\bar{\D}_c=\bar{\D}_{\isome{f}(c)}.\]
        Remarquons que pour toute classe $c\in\orbl$, le sous-graphe $\bar{\D}_c$ est connexe par arcs et de diam\`etre au plus $1,5$ puisque tous les sommets sont deux \`a deux reli\'es par une ar\^ete. Plus g\'en\'eralement \`a tout segment g\'eod\'esique $[c,c']$ de $\orbl$ nous associons un sous graphe d\'efini de la mani\`ere suivante :
        \[ \bar{\D}_{[c,c']}:=\underset{x\in[c,c']}{\bigcup}\bar{\D}_x.\]
        \begin{prop}\label{prop_connexite}
        Pour tous $c,c'\in\orbl$, il existe une suite finie $\{c_i\}_{0\leq i\leq n}$ telle que
        \[\bar{\D}_{[c,c']}=\underset{1\leq i \leq n}{\bigcup}\bar{\D}_{c_i},\]
        o\`u pour tout $0\leq i\leq n-1$ les sous-graphes $\bar{\D}_{c_i}$ et $\bar{\D}_{c_{i+1}}$ sont non-disjoints :
        \[\bar{\D}_{c_i}\cap \bar{\D}_{c_{i+1}}\neq \emptyset.\]
        \end{prop}
        \begin{proof}
        Param\'etrons le segment g\'eod\'esique $[c,c']$ par $\gamma: [0,1]\rightarrow [c,c']$. Construisons une suite $t_i$ de points de $[0,1]$ de la mani\`ere suivante. Soit $t_0=0$. Supposons le point $t_{i-1}$ construit. Si pour tout $t\in[t_{i-1},1]$, $\Som_{\gamma(t)}$ est inclus dans $\Som_{\gamma(t_{i-1})}$, le proc\'ed\'e s'arr\^ete. Sinon, posons :
                \[t_i=\inf \{ t\in \ ]t_{i-1},1] \mid \Som_{\gamma(t)}\nsubseteq\Som_{\gamma(t_{i-1})}\}.\]
                La suite $\gamma(t_i)$ construite satisfait le corollaire \cite[Corollaire 2.13]{Lovoronoi1}. Par cons\'equent, cette suite est finie. Notons $n+1$ le nombre d'\'el\'ements de cette suite. Posons pour tout $0\leq i \leq n$, $c_i=\gamma(t_i)$.
                Nous avons alors :  \[\bar{\D}_{[c,c']}=\underset{1\leq i \leq n}{\bigcup}\bar{\D}_{c_i}.\]
                Montrons \`a pr\'esent que pour tout $0\leq i\leq n-1$, l'intersection entre les sous-graphes $\bar{\D}_{c_i}$ et $\bar{\D}_{c_{i+1}}$ est non-vide.
                Fixons $i$ et consid\'erons le point $c_{i+1}$. D'apr\`es la proposition \cite[Proposition 2.12]{Lovoronoi1}, il existe une constante $\epsilon>0$ d\'ependant de la classe $c_{i+1}$ telle que pour toute application $f$, soit $c_{i+1}$ appartient \`a $\V(f)$, soit la distance entre $c_{i+1}$ et $\V(f)$ est strictement sup\'erieure \`a $\epsilon$. Cela implique que pour toute classe $c$ du segment g\'eod\'esique $[c_i,c_{i+1}[$ \`a distance inf\'erieure ou \'egale \`a $\epsilon$ de $c_{i+1}$, $\bar{\D}_{c}$ est inclus dans $\bar{\D}_{c_{i+1}}$. Comme les cellules de Vorono\"{\i} pavent l'espace $\orbl$, $\bar{\D}_{c}$ est non-vide. Notons $\isome{{f_c}}(\l)$ un sommet de ce sous-graphe. Par d\'efinition des points $\{c_i\}_{0\leq i\leq n}$, ce sommet appartient \'egalement au sous-graphe $\bar{\D}_{c_i}$. Nous avons ainsi montr\'e que l'intersection entre les deux sous-graphes $\bar{\D}_{c_i}$ et $\bar{\D}_{c_{i+1}}$ est non-vide.
        \end{proof}

        Soient $\isome{f}(\l)$ et $\isome{g}(\l)$ deux sommets de $\bar{\D}$. Construisons $\bar{\D}(\isome{f}(\l),\isome{g}(\l))$. Pour cela, consid\'erons le segment g\'eod\'esique associ\'e aux classes $\isome{f}(\l)$ et $\isome{g}(\l)$. Les sommets de $\bar{\D}(\isome{f}(\l),\isome{g}(\l))$ correspondent aux sommets associ\'es aux cellules de Vorono\"{\i} ayant une intersection non vide avec ce segment :
        \[\bar{\D}(\isome{f}(\l),\isome{g}(\l)):=\bar{\D}_{[\isome{f}(\l),\isome{g}(\l)]}.\]

        \begin{cor}\label{cor_connexite_par_arcs}   Pour tous $c,c'\in\orbl$, le sous-graphe $\bar{\D}_{[c,c']}$ est connexe par arcs. En particulier, c'est le cas de $\bar{\D}(\isome{f}(\l),\isome{g}(\l))$ pour tous $f,g\in\Bir(\P^2)$.
        \end{cor}
        \begin{proof}
        Par la proposition \ref{prop_connexite}, il existe une suite finie de classes $\{c_i\}_{0\leq i\leq n}$ telle que :
        \[\bar{\D}_{[c,c']}=\underset{1\leq i \leq n}{\bigcup}\bar{\D}_{c_i}, \] o\`u l'intersection entre $\bar{\D}_{c_i}$ et $\bar{\D}_{c_{i+1}}$ est non-vide pour tout $0\leq i\leq n-1$.
        De plus, chacun des $\bar{\D}_{c_i}$ est connexe par arcs.
        \end{proof}

        \begin{cor}\label{cor_sous-graphe_fini}
            Pour tous $c,c'\in\orbl$, le diam\`etre de $\bar{\D}_{[c,c']}$ est fini.
        \end{cor}

        \subsubsection{Borne pour le diam\`etre du sous-graphe associ\'e \`a un segment de taille fix\'ee}\label{graphe_adjacent_borne_diam_segment_taille_fixee}
        Lors du corollaire \ref{cor_sous-graphe_fini}, nous avons montr\'e que pour toutes classes $c,c'\in\orbl$, le diam\`etre du sous-graphe associ\'e $\bar{\D}_{[c,c']}$ est fini. Nous montrons ici que nous pouvons borner le diam\`etre du sous-graphe $\bar{\D}_{[c,c']}$ dans $\bar{\D}$ en fonction de la longueur du segment $[c,c']$ dans $\H$.

        Nous reprenons les notations introduites dans \cite{Lovoronoi1}. Soit $c=n\l-\underset{p\in\B(\P^2)}{\sum}\lambda_pe_p$ une classe de $\orbl$. Le nombre r\'eel $n$ est sup\'erieur ou \'egal \`a $1$ et est appel\'e le \og degr\'e \fg{} de $c$. Pour tout $p\in\B(\P^2)$, le nombre r\'eel positif ou nul $\lambda_p$ est appel\'e \og la multiplicit\'e \fg{} de $c$ associ\'ee au point $p$. Le \og support \fg{} de $c$ est l'ensemble des points $p$ de $\B(\P^2)$ tels que $\lambda_p$ est non-nulle. Par construction, c'est un ensemble d\'enombrable.

        Notons $p_0$, $p_1$ et $p_2$ trois points de $\B(\P^2)$ de sorte que l'inégalité $\lambda_{p_0}\geq \lambda_{p_1}\geq \lambda_{p_2}\geq \lambda_p$ soit vérifiée pour tout $p\in\B(\P^2)\setminus\{p_0,p_1,p_2\}$. La classe $c$ est appel\'ee \og sp\'eciale \fg{} si les points $p_1$ et $p_2$ sont adh\'erents \`a $p_0$ et si $n-\lambda_{p_0}-\lambda_{p_1}-\lambda_{p_2}<0$. Ceci implique par positivit\'e des exc\`es (D\'efinition \ref{proprietes_c}.\hyperref[propriete_classe_exces]{3}) que $\lambda_{p_0}>\frac{n}{2}$. Dans ce cas, par l'in\'egalit\'e de B\'ezout (D\'efinition \ref{proprietes_c}.\hyperref[propriete_classe_bezout]{4}), $p_0$ est l'unique point de $c$ de multiplicit\'e maximale.

        Dans la suite, nous notons $p_0$ un point du support de $c$ de multiplicit\'e maximale : $\lambda_{p_0}\geq\lambda_p$, pour tout $p\in\B(\P^2)$.
        \begin{lemme}\label{lemme_classe_speciale_mult}
        Soit $c=n\l-\sum_{p\in\B(\P^2)}\lambda_pe_{p}\in \V(\id)$ une classe sp\'eciale. Soient deux points $r$ et $q$ distincts de $p_0$ v\'erifiant $\lambda_r \leq \lambda_q$. Si $p_0$, $q$ et $r$ sont align\'es ou le support d'une application quadratique, alors :
        \[\lambda_r < \frac{n}{4}.\]
        \end{lemme}
        \begin{proof}
        Comme $c$ est une classe sp\'eciale, nous avons d'apr\`es la remarque \cite[Remarque 3.6]{Lovoronoi1} l'in\'egalit\'e : $\lambda_0>\frac{n}{2}$.
        Par hypoth\`ese, $c$ appartient \`a $\V(\id)$ et les points $p_0$, $r$ et $q$ sont soit align\'es, soit forment le support d'une application quadratique, ainsi d'apr\`es la proposition \cite[Proposition 3.5]{Lovoronoi1} nous avons $0\leq n-\lambda_{p_0}-\lambda_r-\lambda_q$. Cela implique que :
            \[2\lambda_r\leq \lambda_q+\lambda_{r}\leq n-\lambda_{p_0}<\frac{n}{2}.\]

\vspace{-0.75cm}
        \end{proof}

Le lemme suivant est le point cl\'e technique qui nous sera utile
pour montrer la Gromov-hyperbolicit\'e du graphe de quasi-adjacence.
Une des difficult\'es vient du fait que la distance entre deux
cellules $\V(f_1)$ et $\V(f_2)$ adjacentes \`a la cellule $\V(\id)$
mais pas quasi-adjacente entres elles, n'est pas forc\'ement
r\'ealis\'ee par la distance entre deux classes
$c_1\in\V(\id)\cap\V(f_1)$ et $c_2\in\V(\id)\cap\V(f_2)$.
\begin{lemme}\label{lemme_nbr_intersection}
Il existe $K>1$ telle que pour toute classe
$c_1\in\V(f_1)\cap\V(\id)$ de degr\'e $n$ et pour toute application
$f_2\in\Bir(\P^2)$ telle que
$\dist_{\bar{\D}}(\isome{\id}(\l),\isome{{f_2}}(\l))=1$ et
$\dist_{\bar{\D}}(\isome{{f_1}}(\l),\isome{{f_2}}(\l))\geq 2$, nous
avons l'in\'egalit\'e :
\[c_1\cdot \isome{{f_2}}(\l)\geq Kn.\]
\end{lemme}

\begin{proof}
Par hypoth\`ese, les cellules $\V(f_1)$ et $\V(f_2)$ ne sont pas
quasi-adjacentes. Ceci implique que les cellules $\V(\id)$,
$\V(f_1)$ et $\V(f_2)$ n'ont pas une classe en commun \`a l'infini.
Le corollaire \cite[Corollaire~5.21]{Lovoronoi1} implique que
\begin{enumerate}
\item[1)] \phantomsection \label{enum_jonq} si les applications $f_1$ et $f_2$ sont des applications de caract\'eristiques Jonqui\`eres alors le ou les (dans le cas quadratique) point(s)-base de multiplicit\'e maximale de ces applications sont disjoints,
\item[2)]  l'union des points-base de $f_1^{-1}$ et de $f_2^{-1}$ :
\begin{enumerate}
\item[a)] \phantomsection \label{enum_card} est ou bien de cardinal sup\'erieur ou \'egal \`a $10$,
\item[b)] \phantomsection \label{enum_align} ou bien contient $4$ points align\'es,
\item[c)] \phantomsection \label{enum_conique} ou bien contient $7$ points sur une conique,
\item[d)] \phantomsection \label{enum_adherents} ou bien contient deux points adh\'erents \`a un troisi\`eme point de cet ensemble.
\end{enumerate}
\end{enumerate}
    Notons $d$ le degr\'e de $f_2$ et $\{m_p\}_{p\in\Bs(f_2^{-1})}$ les multiplicit\'es des points-base de $f_2^{-1}$. Rappelons que \[c_1\cdot \isome{{f_2}}(\l)=dn-\underset{p\in\Bs(f_2^{-1})}{\sum}m_p\lambda_p.\] D'apr\`es les \'equations de Noether, nous avons l'\'egalit\'e $3(d-1)=\underset{p\in\Bs(f_2^{-1})}{\sum}m_p$. Par cons\'equent, nous pouvons partitionner le multi-ensemble compos\'e des points-base de $f_2^{-1}$ compt\'es avec leur multiplicit\'e en $d-1$ triplets de sorte que chaque triplet soit compos\'e de points deux \`a deux disjoints (\cite[Lemme 3.2]{Lovoronoi1}). Notons $T$ une telle partition. Nous avons alors l'\'egalit\'e suivante :
    \begin{equation}\label{eq_triplets}
    c_1\cdot \isome{{f_2}}(\l)=n +\sum\limits_{\{\!\{ p_i,p_j,p_k\}\!\}\in T}(n-\lambda_{p_i}-\lambda_{p_j}-\lambda_{p_k}).
    \end{equation}

Diff\'erencions deux cas selon si la classe $c_1$ est sp\'eciale ou
non.

\begin{enumerate}[leftmargin=0pt, labelwidth=-24pt, listparindent=\parindent]
\item[$\bullet$] Pla\c{c}ons-nous dans le cas o\`u la classe $c_1$ n'est pas sp\'eciale.
\begin{fait}
{\rm S'il existe un point-base $p$ de $f_2^{-1}$ tel que $
\lambda_p\leq\frac{5n}{16}$ alors $c_1\cdot \isome{{f_2}}(\l)\geq
\frac{49}{48}n$.}
\end{fait}
\begin{proof}
    Comme la classe $c_1$ n'est pas sp\'eciale et qu'elle appartient \`a la cellule $\V(\id)$ la somme de ses trois plus grandes multiplicit\'es (not\'ees $\lambda_{p_0}\geq\lambda_{p_1}\geq\lambda_{p_2}$) est inf\'erieure ou \'egale \`a son degr\'e d'apr\`es \cite[Corollaire 3.7]{Lovoronoi1}. Par cons\'equent, pour tout triplet $\{\!\{ p_i,p_j,p_k\}\!\}\in T$ nous avons :
    \[n-\lambda_{p_i}-\lambda_{p_j}-\lambda_{p_k}\geq n-\lambda_{p_0}-\lambda_{p_1}-\lambda_{p_2}\geq 0.\]
    Ainsi, pour tout triplet $\{\!\{ p_i,p_j,p_k\}\!\}\in T$ l'in\'egalit\'e suivante est satisfaite :
    \begin{equation}\label{eq_c1f2_triplet}
    c_1\cdot \isome{{f_2}}(\l)\geq n + (n-\lambda_{p_i}-\lambda_{p_j}-\lambda_{p_k}).
    \end{equation}
    Par hypoth\`ese, il existe un point-base de $f_2^{-1}$ tel que sa multiplicit\'e pour $c$ soit inf\'erieure ou \'egale \`a $\frac{5n}{16}$. Notons ce point $q$.
    Consid\'erons deux cas selon si $f_1$ est une application de caract\'eristique Jonqui\`eres ou pas.
    \begin{enumerate}[leftmargin=0pt, labelwidth=-24pt]
        \item[$*$] Si $f_1$ n'est pas de caract\'eristique Jonqui\`eres, d'apr\`es le th\'eor\`eme \cite[Th\'eor\`eme 4.1]{Lovoronoi1} toutes les multiplicit\'es de $c$ sont inf\'erieures ou \'egales \`a $\frac{n}{3}$. Par cons\'equent, en consid\'erant un triplet de $T$ o\`u le point $q$ appara\^{\i}t nous avons d'apr\`es l'\'equation (\ref{eq_c1f2_triplet}) la minoration suivante : \[c_1\cdot \isome{{f_2}}(\l)\geq n+ (n-2\frac{n}{3}-\frac{5n}{16})=\frac{49}{48}n .\]
        \item[$*$] Si $f_1$ est de caract\'eristique Jonqui\`eres de degr\'e sup\'erieur ou \'egal \`a $3$, tout point distinct du point-base maximal $p_0$ de $f_1^{-1}$ a une multiplicit\'e pour $c$ inf\'erieure ou \'egale \`a $\frac{n}{3}$ (voir \cite[Th\'eor\`eme 4.1]{Lovoronoi1}). De plus, par le point \hyperref[enum_jonq]{1)}, $f_2$ ne peut pas \^etre une application de caract\'eristique Jonqui\`eres en ce point-l\`a et ainsi $m_0<d-1$. Ceci signifie qu'il existe un triplet $t_1$ ne contenant pas le point $p_0$. Quitte \`a \'echanger le point $q$ avec un point du triplet $t_1$ qui n'appara\^{\i}t pas dans un triplet contenant le point $q$, nous pouvons supposer qu'il existe un triplet contenant le point $q$ et ne contenant pas le point $p_0$. Nous obtenons la m\^eme minoration que dans le cas pr\'ec\'edent.
        Si $f_1$ est une application quadratique alors par \cite[Th\'eor\`eme~4.1]{Lovoronoi1} toutes les multiplicit\'es de $c$ autres que $\lambda_{p_0}$ et $\lambda_{p_1}$ sont inf\'erieures ou \'egales \`a $\frac{n}{3}$. De plus, par hypoth\`ese, si $f_2$ est une application de caract\'eristique Jonqui\`eres, le (ou les) point-base de multiplicit\'e maximale de $f_2$ n'est ni $p_0$ ni $p_1$ ni $p_2$. Ainsi, par le m\^eme argument que pr\'ec\'edemment, il existe un triplet $t_1$ ne contenant pas le point $p_0$. Si ce triplet contient le point $p_1$, comme $m_1<d-1$ il existe un triplet $t'$ diff\'erent ne contenant pas le point $p_1$. Nous \'echangeons alors un point du triplet $t'$ diff\'erent de $p_0$ avec le point $p_1$ qui se trouve dans le triplet $t_1$. De m\^eme que pr\'ec\'edemment, si $t_1$ ne contient pas le point $q$ nous pouvons \'echanger $q$ avec un des trois points de $t_1$ qui n'apparaît pas dans un triplet contenant $q$. Ainsi tous les triplets sont positifs et un triplet contient $q$ mais pas $p_0$ et $p_1$ et nous obtenons la m\^eme majoration que dans le cas pr\'ec\'edent.
    \end{enumerate}
\vspace{-0,6cm}
\end{proof}

Il nous reste maintenant \`a montrer qu'il existe toujours un
point-base de $f_2^{-1}$ dont la multiplicit\'e associ\'ee est
inf\'erieure ou \'egale \`a $\frac{5n}{16}$.
\begin{fait}
{\rm Si $\lambda_{p_0}\geq \frac{3n}{8}$ alors il existe un
point-base $p$ de $f_2^{-1}$ tel que $ \lambda_p\leq\frac{5n}{16}$.}
\end{fait}
\begin{proof}
Dans ce cas, par le th\'eor\`eme \cite[Th\'eor\`eme
4.1]{Lovoronoi1}, $f_1$ est une application de caract\'eristique
Jonqui\`eres dont l'inverse a pour base maximal le point $p_0$. Si
$f_1$ est de degr\'e sup\'erieur ou \'egal \`a $3$, toujours
d'apr\`es le m\^eme th\'eor\`eme nous obtenons la majoration
attendue pour tout point $p\in\B(\P^2)$ diff\'erent de $p_0$ :
\[\lambda_p\leq \frac{n-\lambda_{p_0}}{2}\leq \frac{5n}{16}.\]
Si $f_1$ est de degr\'e $2$, d'apr\`es \cite[Th\'eor\`eme
4.1]{Lovoronoi1}, nous avons $
n=\lambda_{p_0}+\lambda_{p_1}+\lambda_{p_2}$. Ainsi, pour tout point
$p\in\B(\P^2)$ diff\'erent de $p_0$ et de $p_1$ nous obtenons :
\[\lambda_p\leq\lambda_{p_2}\leq\frac{\lambda_{p_2}+\lambda_{p_1}}{2} =\frac{n-\lambda_{p_0}}{2}\leq\frac{5n}{16}.\]
Par cons\'equent, tout point-base de $f_2^{-1}$ diff\'erent de $p_0$
et de $p_1$ poss\`ede la propri\'et\'e attendue.
\end{proof}

D'apr\`es \cite[Th\'eor\`eme 4.1]{Lovoronoi1}, comme la classe $c_1$
n'est pas sp\'eciale, nous avons :
\begin{equation}\label{eq_mult_f1_f2}
\lambda_q\geq\lambda_p \text{ pour tout } q\in\Bs(f_1^{-1}) \text{
et pour tout }p\notin\Bs(f_1^{-1}).
\end{equation}
Nous avons besoin de la variante suivante en consid\'erant des
points-base de $f_2^{-1}$.

\begin{fait}\label{fait_pt_min_f_2}
{\rm    Il existe un point-base de $f_2^{-1}$,  $p\in\Bs(f_2^{-1})$,
tel que pour tout point-base $q\in\Bs(f_1^{-1})$, les
multiplicit\'es associ\'ees satisfont : $\lambda_p\leq \lambda_q$. }
\end{fait}
\begin{proof}
    Si $\Bs(f_2^{-1})$ n'est pas inclus dans $\Bs(f_1^{-1})$ alors il existe $p\in \Bs(f_2^{-1})$ tel que $p\notin \Bs(f_1^{-1})$. Ainsi, d'apr\`es l'in\'egalit\'e (\ref{eq_mult_f1_f2}), pour tout point $q\in\Bs(f_1^{-1})$, $\lambda_p\leq \lambda_q$ et le fait est prouv\'e dans ce cas.

    Sinon, nous distinguons deux cas suivant que $f_1$ est de caract\'eristique Jonqui\`eres ou non. Si $f_1$ n'est pas de caract\'eristique Jonqui\`eres alors d'apr\`es \cite[Th\'eor\`eme 4.1]{Lovoronoi1}, pour tout point $p\in \Bs(f_2^{-1})$, $\lambda_p=\frac{n}{3}=\lambda_q$ pour tout point $q\in\Bs(f_1^{-1})$ ce qui prouve le fait dans ce cas.
    Supposons maintenant que $f_1^{-1}$ est une application de caract\'eristique Jonqui\`eres. Si elle est de degr\'e $2$, comme nous sommes dans le cas o\`u les points-base de $f_2^{-1}$ sont inclus dans ceux de $f_1^{-1}$,  $f_2^{-1}$ poss\`ede les m\^emes trois points-base et est une application quadratique. Or par hypoth\`ese \ref{enum_jonq}, ceci est impossible.
    L'application $f_1^{-1}$ est par cons\'equent de degr\'e sup\'erieur ou \'egal \`a trois. Comme $f_2^{-1}$ poss\`ede au moins trois points-base, il en existe un diff\'erent de $p_0$, not\'e $p$ et par \cite[Th\'eor\`eme 4.1]{Lovoronoi1}, $\lambda_p=\frac{n-\lambda_0}{2}=\lambda_q\leq \lambda_{p_0}$ pour tout point $q\in\Bs(f_1^{-1})$ diff\'erent de $p_0$ ce qui finit de démontrer le fait.
\end{proof}

Dans le cas o\`u $\lambda_{p_0}$ est strictement inf\'erieur \`a
$\frac{3n}{8}$, nous distinguons les $4$ possibilit\'es :
\hyperref[enum_card]{a)}, \hyperref[enum_align]{b)},
\hyperref[enum_conique]{c)}, \hyperref[enum_adherents]{d)}, et
montrons qu'il existe un point-base de $f_2^{-1}$ dont la
multiplicit\'e associ\'ee est inf\'erieure ou \'egale \`a
$\frac{5n}{16}$.

\begin{fait}
{\rm Si $\lambda_{p_0}< \frac{3n}{8}$ et si nous sommes dans le cas
\hyperref[enum_card]{a)} alors il existe un point-base $p$ de
$f_2^{-1}$ tel que $ \lambda_p\leq\frac{5n}{16}$.}
\end{fait}
\begin{proof}
Notons $\lambda_{\min}$ la plus petite multiplicit\'e pour $c$
correspondant \`a un point-base de l'union des points-base de
$f_1^{-1}$ et de $f_2^{-1}$. Comme l'union des points-base de
$f_1^{-1}$ et de $f_2^{-1}$ est de cardinal sup\'erieur ou \'egal
\`a $10$, nous avons alors par l'in\'egalit\'e contre
l'anti-canonique
\ref{proprietes_c}.\hyperref[propriete_classe_canonique]{2)} :
\[10\lambda_{\min}\leq \sum\limits_{p\in\Bs(f_1^{-1})\cup\Bs(f_2^{-1})}\lambda_p\leq\sum\limits_{p\in\B(\P^2)}\lambda_p\leq 3n.\]
Ainsi, $\lambda_{\min}\leq \frac{3n}{10}\leq \frac{5n}{16}$. Si
cette multiplicit\'e n'est pas associ\'ee \`a un point-base de
$f_2^{-1}$, elle correspond donc \`a un point-base de $f_1^{-1}$ et
par le fait \ref{fait_pt_min_f_2}, il existe $p\in \Bs(f_2^{-1})$
dont la multiplicit\'e associ\'ee est inf\'erieure \`a
$\lambda_{\min}$: $\lambda_{p}\leq \lambda_{\min}\leq\frac{5n}{16}$.
\end{proof}
\begin{fait}
{\rm Si $\lambda_{p_0}< \frac{3n}{8}$ et si nous sommes dans le cas
\hyperref[enum_align]{b)} alors il existe un point-base $p$ de
$f_2^{-1}$ tel que $ \lambda_p\leq\frac{5n}{16}$.}
\end{fait}
\begin{proof}
Si $4$ des points de l'union des points-base de $f_1^{-1}$ et de
$f_2^{-1}$ sont align\'es alors en notant $\lambda_{\min}$ la plus
petite des $4$ multiplicit\'es et $p_{\min}$ le point associ\'e,
nous avons par l'in\'egalit\'e de B\'ezout
\ref{proprietes_c}.\hyperref[propriete_classe_bezout]{4)} :
\[4\lambda_{\min}\leq n.\]
Par cons\'equent, $\lambda_{\min}\leq \frac{n}{4}\leq\frac{5n}{16}$.
De plus, comme pr\'ec\'edemment, si $p_{\min}\notin \Bs(f_2^{-1})$,
par le fait \ref{fait_pt_min_f_2}, il existe un point $p\in
\Bs(f_2^{-1})$ tel que $\lambda_p\leq \lambda_{\min}\leq
\frac{5n}{16}$ ce qui ach\`eve la preuve.
\end{proof}
Le cas o\`u $7$ points sont sur une conique se prouve de fa\c{c}on
similaire.
\begin{fait}
{\rm Si $\lambda_{p_0}< \frac{3n}{8}$ et si nous sommes dans le cas
\hyperref[enum_conique]{c)} alors il existe un point-base $p$ de
$f_2^{-1}$ tel que $ \lambda_p\leq\frac{5n}{16}$.}
\end{fait}
\begin{fait}
{\rm Si $\lambda_{p_0}< \frac{3n}{8}$ et si nous sommes dans le cas
\hyperref[enum_adherents]{d)} alors il existe un point-base $p$ de
$f_2^{-1}$ tel que $ \lambda_p\leq\frac{5n}{16}$.}
\end{fait}
\begin{proof}
Consid\'erons le cas o\`u $2$ points $q_1$ et $q_2$ de l'ensemble de
l'union de points-base de $f_1^{-1}$ et de $f_2^{-1}$ sont
adh\'erents \`a un troisi\`eme point $q_0$. En notant
$\lambda_{\min}$ la plus petite des deux multiplicit\'es associ\'ees
\`a ces points pour $c$, nous avons par positivit\'e des exc\`es :
\[2\lambda_{\min}\leq \lambda_{q_1}+\lambda_{q_2}\leq \lambda_{q_0}\leq \frac{3n}{8}.\]
Comme pr\'ec\'edemment, si le point dont la multiplicit\'e est
$\lambda_{\min}$ n'est pas un point-base de $f_2^{-1}$ alors c'est
un point-base de $f_1^{-1}$ et nous pouvons conclure par le fait
\ref{fait_pt_min_f_2}.
\end{proof}
Ceci ach\`eve la preuve du lemme \ref{lemme_nbr_intersection} dans
le cas o\`u la classe $c_1$ n'est pas sp\'eciale.

\item[$\bullet$] Int\'eressons-nous maintenant au cas o\`u la classe $c_1$ est sp\'eciale. Par le th\'eor\`eme \cite[Th\'eor\`eme~4.1]{Lovoronoi1} l'application $f_1$ est une application de caract\'eristique Jonqui\`eres et le point-base maximal de $f_1^{-1}$ est $p_0$. De plus, notons $p_1$ et $p_2$ deux points adh\'erents au point $p_0$ tels que $\lambda_{p_0}\geq\lambda_{p_1}\geq \lambda_{p_2}\geq \lambda_p$ pour tout $p\in\B(\P^2)\setminus\{p_0,p_1,p_2\}$.

Comme $f_2$ est le germe d'une cellule quasi-adjacente \`a $\V(\id)$
soit elle est de caract\'eristique Jonqui\`eres soit elle ne
poss\`ede pas deux points-base adh\'erents \`a un m\^eme troisi\`eme
de ses points-base (Th\'eor\`eme~\ref{cor_cellules_quasi_adjacentes}). Si elle est de
caract\'eristique Jonqui\`eres, d'apr\`es \ref{enum_jonq}, le point
$p_0$ n'est pas le point-base de multiplicit\'e maximal de
$f_2^{-1}$ (ou l'un des dans le cas o\`u $f_2$ est quadratique). Par
cons\'equent, $f_2^{-1}$ ne peut pas poss\'eder deux points-base
adh\'erents au point $p_0$. Ainsi, quelle que soit $f_2$, l'application
$f_2^{-1}$ ne poss\`ede pas deux points adh\'erents au point $p_0$.

Parmi les points-base de $f_2^{-1}$ diff\'erents de $p_0$ (si ce
dernier appartient \`a $\Bs(f_2^{-1})$), consid\'erons-en deux de
multiplicit\'es maximales pour $c$. Notons-les $q$ et $r$ avec
$\lambda_{q}\geq\lambda_{r}$. Comme ces deux points ne sont pas
adh\'erents au point $p_0$ et par positivit\'e des exc\`es pour $c$
\ref{proprietes_c}.\hyperref[propriete_classe_exces]{3)}, les trois
points $p_0$, $q$ et $r$ sont align\'es ou forment le support d'une
application quadratique. Comme $c$ appartient \`a $V(\id)$ nous
obtenons par la proposition \cite[Proposition 3.5]{Lovoronoi1} que
pour tout triplet $(p_i,p_j,p_k)$ de points-base de $f_2^{-1}$ ne
contenant pas deux fois le point $p_0$ ni deux fois le point $q$ :
\begin{equation}\label{pt_grand_f_2}
n-\lambda_{p_i}-\lambda_{p_j}-\lambda_{p_k}\geq
n-\lambda_{p_0}-\lambda_{q}-\lambda_{r}\geq 0.
\end{equation}
De plus, par le lemme \ref{lemme_classe_speciale_mult},
$\lambda_{r}\leq\frac{n}{4}$ et c'est le cas de toutes les
multiplicit\'es pour $c$ associ\'ees aux points-base de $f^{-1}_2$
hors $p_0$ et $p_{q}$.

Si l'application $f^{-1}_2$ n'est pas une application de
caract\'eristique Jonqui\`eres de point-base maximal le point $q$
alors les multiplicit\'es $m_{p_0}$ et $m_{q}$ sont strictement
inf\'erieures \`a $d-1$. Par cons\'equent, nous pouvons partitionner
le multi-ensemble des points-base de $f_2^{-1}$ compt\'es avec
multiplicit\'e en $d-1$ triplets de sorte qu'un triplet ne contienne
pas $p_0$ et $q$ et que tous les autres contiennent au plus une fois
le point $p_0$ et une fois le point $q$. Par l'\'equation
\eqref{pt_grand_f_2}, les termes de la somme \eqref{eq_triplets}
associ\'es \`a ces triplets sont tous positifs. De plus, le terme
associ\'e au triplet ne contenant pas les points $p_0$ et $q$ est
minor\'e par : $n-\frac{3n}{4}=\frac{n}{4}$. Ainsi, \[c_1\cdot
\isome{{f_2}}(\l)\geq \frac{5}{4}n.\]

Si maintenant le point $q$ est un point-base maximal de
l'application $f_2^{-1}$ de caract\'eristique Jonqui\`eres alors le
point $q$ est dans $\P^2$ et les points $p_0$, $p_1$ et $q$ sont
soit align\'es soit le support d'une application quadratique. Le
lemme \ref{lemme_classe_speciale_mult} implique que $\lambda_{q}\leq
\frac{n}{4}$ et c'est le cas de tous les points-base de $f^{-1}_2$
except\'e $p_0$ si c'est un point-base de $f_2^{-1}$. Comme
pr\'ec\'edemment, la multiplicit\'e $m_{p_0}$ est strictement
inf\'erieure \`a $d-1$ et il existe donc au moins un triplet o\`u le
point $p_0$ n'appara\^{\i}t pas. Par cons\'equent, le terme de la
somme \eqref{eq_triplets} correspondant au triplet o\`u le point
$p_0$ n'appara\^{\i}t pas est minor\'e par :
$n-\frac{3n}{4}=\frac{n}{4}$. Tous les autres triplets sont
constitu\'es d'indices deux \`a deux disjoints ainsi d'apr\`es
l'in\'egalit\'e \eqref{pt_grand_f_2} ils sont tous positifs ou nuls.
Par cons\'equent dans ce cas-l\`a nous obtenons la m\^eme minoration
\[c_1\cdot \isome{{f_2}}(\l)\geq \frac{5}{4}n.\] Ceci ach\`eve la
preuve dans le cas o\`u la classe $c_1$ est sp\'eciale et par
cons\'equent celle du lemme \ref{lemme_nbr_intersection}.
\end{enumerate}

\vspace{-0,67cm}
\end{proof}

        \begin{lemme}\label{lemme_dist_petite_nb_inter}
        Soient $c_1\in\V(\id)$, $c_2\in\orbl$ et $\epsilon>0$ tels que $\dist(c_1,c_2)\leq \epsilon$.
        Alors, pour toute application $f_2\in\Bir(\P^2)$ telle que $c_2\in\V(f_2)$, nous avons l'in\'egalit\'e :
       \[0\leq1-\dfrac{c_1\cdot\l}{c_1\cdot \isome{{f_2}}(\l)}\leq 2\epsilon.\]
        \end{lemme}

        \begin{proof}
        Par la d\'efinition des cellules de Vorono\"{\i}, nous avons d'une part \[\dist(c_1,\l)\leq\dist(c_1,\isome{{f_2}}(\l)) \text{ et }  \dist(c_2,\isome{{f_2}}(\l))\leq\dist(c_2,\l)\] et d'autre part, par l'in\'egalit\'e triangulaire nous obtenons :
        \begin{equation}\label{dist_c1_f_2}
        \dist(c_1,\isome{{f_2}}(\l)) \leq \epsilon+ \dist(c_2,\isome{{f_2}}(\l))\leq \epsilon+\dist(c_2,\l)\leq 2\epsilon +\dist(c_1,\l).
        \end{equation}

         Par cons\'equent nous avons l'encadrement suivant :
         \begin{equation*}
        0\leq \dist(c_1,\isome{{f_2}}(\l))-\dist(c_1,\l)\leq 2\epsilon.
         \end{equation*}

        En utilisant la majoration pr\'ec\'edente et par concavit\'e de la fonction $\argcosh$ nous obtenons l'in\'egalit\'e annonc\'ee :
        \begin{align*}
        2\epsilon \geq \dist(c_1,\isome{{f_2}}(\l))-\dist(c_1,\l)&=\argcosh(c_1\cdot\isome{{f_2}}(\l))-\argcosh(c_1\cdot\l)\\
        &\geq \frac{c_1\cdot\isome{{f_2}}(\l)-c_1\cdot\l}{\sqrt{(c_1\cdot\isome{{f_2}}(\l))^2-1}}\\
        &\geq \frac{c_1\cdot\isome{{f_2}}(\l)-c_1\cdot\l}{c_1\cdot\isome{{f_2}}(\l)}= 1-\frac{c_1\cdot\l}{c_1\cdot\isome{{f_2}}(\l)}.
        \end{align*}

\vspace{-0.9cm}
        \end{proof}

\medskip

Soient $E_1$ et $E_2$ deux sous-ensembles de $\orbl$. Nous
d\'efinissons :
\[\dist(E_1,E_2)=\inf\{\dist(c_1,c_2)\mid x_1\in E_1, \ x_2\in E_2\}.\]
Attention, bien que souvent appel\'ee \og distance \fg{} entre $E_1$
et $E_2$ ce n'est pas une distance sur les sous-ensembles de
$\orbl$.

\begin{lemme}\label{lemme_distance_cellules_disjointes}
    Il existe $M>0$ tel que pour toutes applications $f_1,f_2\in\Bir(\P^2)$ telles que $\dist_{\bar{\D}}(\isome{{f_1}}(\l),\isome{{f_2}}(\l))= 2$, nous avons :
    \[\dist(\V(f_1),\V(f_2))\geq M.\]
\end{lemme}

\begin{proof}
Soient $c\in\V(f_1)$ et $c'\in\V(f_2)$. Montrons que
$\dist(c,c')\geq M$. Consid\'erons le segment $[c,c']$. Quitte \`a
faire agir le germe de la cellule travers\'ee juste apr\`es
$\V(f_1)$, nous pouvons supposer que c'est la cellule associ\'ee \`a
l'identit\'e et que la cellule $\V(f_1)$ est adjacente \`a la
cellule $\V(\id)$. Notons $c_1$ la classe du segment $[c,c']$ qui se
trouve dans l'intersection $\V(\id)\cap\V(f_1)$ et $n$ son degr\'e.

Consid\'erons dans un premier temps le cas o\`u la cellule $\V(f_2)$
est quasi-adjacente \`a $\V(\id)$. Notons $c_2$ la classe de
l'intersection $[c,c']\cap\V(f_2)$ qui minimise la distance \`a $c$
:
\[\dist(c,c_2)\leq \dist(c,y) \text{ pour tout }  y\in\V(f_2)\cap[c,c']. \]
Reprenons la constante $K>1$ du lemme \ref{lemme_nbr_intersection}
et posons $M=\frac{K-1}{3K}$. Montrons que la distance entre $c_1$
et $c_2$ est sup\'erieure ou \'egale \`a $M$. Supposons le contraire
alors par le lemme \ref{lemme_dist_petite_nb_inter}, nous obtenons :
\[1-\frac{c_1\cdot\l}{c_1\cdot \isome{{f_2}}(\l)}\leq 2M.\] Or
d'apr\`es le lemme \ref{lemme_nbr_intersection}, nous avons $ Kn\leq
c_1\cdot \isome{{f_2}}(\l)$, ce qui implique la contradiction $\frac{K-1}{K}\leq 2M =\frac{2(K-1)}{3K}$.

Pla\c{c}ons-nous maintenant dans le cas o\`u la cellule $\V(f_2)$
n'est pas quasi-adjacente \`a $\V(\id)$. Montrons qu'il existe trois
cellules travers\'ees par le segment $[c,c']$ ayant la configuration
pr\'ec\'edente. Consid\'erons les cellules travers\'ees par le
segment $[c,c']$ et consid\'erons le premier germe $g_1$ tel que les
cellules $\V(f_1)$ et $\V(g_1)$ soient \`a distance $2$ dans le
graphe de quasi-adjacence
$\dist_{\bar{\D}}(\isome{{f_1}}(\l),\isome{{g_1}}(\l))= 2$. Par
cons\'equent, celle juste avant not\'ee $\V(g)$ est \`a distance $1$
de $\V(f_1)$. Alors, les cellules $\V(g)$ et $\V(g_1)$ sont
adjacentes entre elles et $\V(g)$ et $\V(f_1)$ sont quasi-adjacentes
entre elles. Nous nous retrouvons dans le cas pr\'ec\'edent o\`u $g$
remplace $\id$, $g_1$ remplace $f_1$ et $f_1$ remplace $f_2$. Ceci
ach\`eve la preuve du lemme.
\end{proof}

\begin{prop}\label{prop_diam_borne_segment}
    Pour toute constante $L>0$, il existe $B(L)>0$ tel que pour tous $c,c'\in\orbl$ v\'erifiant $\dist(c,c')\leq L$, le diam\`etre du graphe associ\'e au segment est born\'e par $B(L)$ : \[\diam(\bar{\D}_{[c,c']})\leq B(L).\]
\end{prop}

\begin{proof}
    Param\'etrons le segment $[c, c']$ par $\gamma:[0,1]\rightarrow [c,c']$ avec $\gamma(0) = c$ et $\gamma(1) = c'$.
    Nous construisons une suite de points $\{t_i\}$ de $[0,1]$, et une suite d'applications $f_i \in \Bir(\P^2)$, par le proc\'ed\'e de r\'ecurrence suivant. Initialisons en posant $t_0 = 0$, et en choisissant $f_0 \in \Bir(\P^2)$ tel que $c = \gamma(0) \in \V(f_0)$.
    Pour $i \geq 1$, supposons le point $t_{i-1} \in [0,1[$ construit et l'application $f_{i-1}$ choisie. Si pour tout $t\in[t_{i-1},1]$, toute application $f\in\Bir(\P^2)$ telle que $\gamma(t)\in\V(f)$ satisfait $\dist_{\bar{\D}}(\isome{{f_{i-1}}}(\l),\isome{{f_{i}}}(\l))\leq1$, alors nous posons $t_i=1$ et le proc\'ed\'e s'arr\^ete. Sinon nous posons : \[ t_i = \inf \Big\{t \in ]t_{i-1},1]\mid \exists f \in \Bir(\P^2), \gamma(t) \in \V(f) \text{ et } \dist_{\bar \D}(\isome{{f_{i-1}}}(\l), \isome{{f_i}}(\l))=2\Big\}.\]
    Ensuite, nous choisissons $f_i \in \Bir(\P^2)$ satisfaisant $\dist_{\bar \D}(\isome{{f_{i-1}}}(\l), \isome{{f_i}}(\l))=2$ et $\gamma(t_i) \in \V(f_i)$.
    Si $t_i < 1$, le lemme \ref{lemme_distance_cellules_disjointes} implique qu'il existe une constante universelle $M > 0$ telle que $d(\gamma(t_i), \gamma(t_{i-1})) \ge M$.
    En cons\'equence, apr\`es un nombre fini d'\'etapes le proc\'ed\'e s'arr\^ete.
    Plus pr\'ecis\'ement, la suite poss\`ede au plus $\lceil \frac{L}{M}\rceil+1$ termes.
    Posons pour tout $i$, $c_i=\gamma(t_i)$. Pour tout $i\geq 0$ les sommets du sous-graphe $\bar{\D}_{[c_i,c_{i+1}[}$ sont par construction \`a distance $1$ du sommet $\bar{f_i}$. Ainsi, le diam\`etre de ce sous-graphe est born\'e :
    \[\diam(\bar{\D}_{[c_i,c_{i+1}[})\leq 1,5.\]
    Nous obtenons finalement : \[\diam(\bar{\D}_{[c,c']})\leq \sum\limits_{i=0}^{\lceil \frac{L}{M}\rceil-1}\diam(\bar{\D}_{[c_i,c_{i+1}[})+\diam(\bar{\D}_{c'})\leq 1,5\times\bigg(\bigg\lceil \frac{L}{M}\bigg\rceil+1\bigg),\]
    ce qui ach\`eve la preuve de la proposition en posant $B(L)=1,5\times (\lceil \frac{L}{M}\rceil+1)$.
\end{proof}

    \subsubsection{Hyperbolicit\'e du graphe de quasi-adjacence}\label{subsubsection_hyperbolicite_adjacent}
    Nous pouvons maintenant appliquer le crit\`ere de Bowditch (Th\'eor\`eme \ref{thm_bow}) pour obtenir notre r\'esultat principal.

    \begin{thm}
        Le graphe de quasi-adjacence est hyperbolique au sens de Gromov.
    \end{thm}
    \begin{proof}
    Soit $h=\max(B(\delta),B(17))$ o\`u $\delta=\ln(1+\sqrt{2})$ est la constante de Gromov-hyperbolicit\'e de $\H$ et $B$ est l'application de la proposition \ref{prop_diam_borne_segment}. \`A tous sommets $\isome{f}(\l)$ et $\isome{g}(\l)$ nous associons les sous-graphes construits dans la sous-section \ref{subsubsection_construction_sous_graphes} : \[\bar{\D}(\isome{f}(\l),\isome{g}(\l)).\]
        Par construction, pour tous $f,g\in\Bir(\P^2)$, les sommets $\isome{f}(\l)$ et $\isome{g}(\l)$ appartiennent au sous-graphe $\bar{\D}(\isome{f}(\l),\isome{g}(\l))$. De plus d'apr\`es le corollaire \ref{cor_connexite_par_arcs}, ce sous-graphe $\bar{\D}(\isome{f}(\l),\isome{g}(\l))$ est connexe par arcs. Ceci prouve le point \hyperref[critere_bow1]{1)} du crit\`ere.

        Montrons que le point \hyperref[critere_bow2]{2)} est v\'erifi\'e. Consid\'erons trois sommets $\isome{{f_1}}(\l)$, $\isome{{f_2}}(\l)$ et $\isome{{f_3}}(\l)$ du graphe de quasi-adjacence. Soit $\isome{{f}}(\l)$ un sommet de $\bar{\D}(\isome{{f_1}}(\l),\isome{{f_2}}(\l))$. Par construction, il existe une classe $c_{f}\in[\isome{{f_1}}(\l),\isome{{f_2}}(\l)]$ telle que $\isome{{f}}(\l)$ est un sommet du sous-graphe $\bar{\D}_{c_{f}}$ associ\'e \`a $c_f$.
        L'espace $\orbl$ \'etant $\delta$-hyperbolique, il existe une classe $c_2\in [\isome{{f_1}}(\l),\isome{{f_3}}(\l)]\cup [\isome{{f_2}}(\l),\isome{{f_3}}(\l)]$ telle que $\dist(c_2,c_f)\leq  \delta$.
        Pour tout sommet $\isome{{g}}(\l)\in\bar{\D}_{c_2}\subset\bar{\D}(\isome{{f_2}}(\l),\isome{{f_3}}(\l))\cup\bar{\D}(\isome{{f_1}}(\l),\isome{{f_3}}(\l))$, nous avons d'apr\`es la proposition \ref{prop_diam_borne_segment} : \[\dist_{\bar{\D}}(\isome{{f}}(\l),\isome{{g}}(\l))\leq \diam(\bar{\D}_{[c_f,c_2]})\underset{\ref{prop_diam_borne_segment}}{\leq} B(\delta).\]
        Ceci prouve le point \hyperref[critere_bow2]{2)}.

        Int\'eressons-nous \`a pr\'esent au point \hyperref[critere_bow3]{3)}.
        Soient $\isome{{f_1}}(\l)$ et $\isome{{f_2}}(\l)$ deux sommets de $\bar{\D}_c$ \`a distance inf\'erieure ou \'egale \`a $1$. Si les deux sommets sont identiques, la condition \hyperref[critere_bow3]{3)} est imm\'ediatement v\'erifi\'ee puisque le sous-graphe associ\'e \`a ces deux sommets est r\'eduit au sommet $\isome{{f_1}}(\l)$. Supposons \`a pr\'esent que les deux sommets $\isome{{f_1}}(\l)$ et $\isome{{f_2}}(\l)$ sont \`a distance $1$. Nous voulons montrer que $\bar{\D}(\isome{{f_1}}(\l),\isome{{f_2}}(\l))$ est de diam\`etre born\'e. Quitte \`a faire agir $f^{-1}_2$, nous pouvons supposer que $f_2=\id$ et que $f_1$ est le germe d'une cellule quasi-adjacente \`a $\V(\id)$. D'apr\`es le corollaire \ref{cor_cellules_quasi_adjacentes}, l'application $f_1$ est soit une application de caract\'eristique Jonqui\`eres, soit elle n'est pas de caract\'eristique Jonqui\`eres mais son inverse poss\`ede au plus $8$ points-base en position presque g\'en\'erale, soit ce n'est pas une application de caract\'eristique Jonqui\`eres et son inverse poss\`ede exactement $9$ points-base en position presque g\'en\'erale. Nous distinguons ces trois cas.
        \begin{enumerate}[leftmargin=0pt, labelwidth=-24pt, listparindent=\parindent]
            \item[$\bullet$] Si l'application $f_1$ est de caract\'eristique Jonqui\`eres alors par le lemme \cite[Lemme 4.11]{Lovoronoi1}, nous avons l'inclusion : \[[\l,\isome{{f_1}}(\l)]\subset \V(\id)\cup\V(f_1)\] qui implique que  \[\diam\Big(\bar{\D}(\isome{{f_1}}(\l),\isome{\id}(\l))\Big)\leq 1,5.\]

            \item[$\bullet$] Consid\'erons le cas o\`u l'application $f_1$ n'est pas de caract\'eristique Jonqui\`eres et son inverse poss\`ede au plus $8$ points-base en position presque g\'en\'erale. C'est donc le germe d'une cellule adjacente \`a la cellule $\V(\id)$ (Corollaire \ref{cor_cellule_adjacente}).
            D'apr\`es \cite[Corollaire 4.9]{Lovoronoi1} le degr\'e de $f_1$ est inf\'erieur ou \'egal \`a $17$, ainsi la longueur du segment g\'eod\'esique $[\l,\isome{{f_1}}(\l)]$ est inf\'erieure ou \'egale \`a $17$. La proposition \ref{prop_diam_borne_segment} implique que le diam\`etre de $\bar{\D}(\isome{{f_1}}(\l),\isome{{\id}}(\l))$ est inf\'erieur ou \'egal \`a $B(17)$.
            Comme $h\geq B(17)$, le point \hyperref[critere_bow3]{3)} du crit\`ere de Bowditch est v\'erifi\'e dans ce cas-l\`a.
            \item[$\bullet$] Consid\'erons le dernier cas, celui o\`u l'application $f_1$ n'est pas de caract\'eristique Jonqui\`eres et son inverse  poss\`ede exactement $9$ points-base en position presque g\'en\'erale. C'est le germe d'une cellule quasi-adjacente non adjacente \`a $\V(\id)$ (Corollaire \ref{cor_cellules_quasi_adjacentes}). Le segment reliant $\l$ et $\isome{{f_1}}(\l)$ est constitu\'e de classes ayant exactement les $9$ points-base de $f_1^{-1}$ comme support. En particulier, les classes de ce segment g\'eod\'esique ne sont pas sp\'eciales.
            Montrons que pour tout sommet $\isome{{g}}(\l)$ de $\bar{\D}_{[\l,\isome{{f_1}}(\l)]}$ la cellule $\V(g)$ poss\`ede dans son bord \`a l'infini la classe \[s=3\l-\sum\limits_{p\in\Bs(f_1^{-1})}e_p\] et est par cons\'equent \`a distance $1$ du sommet $\isome{{\id}}(\l)$ dans le graphe de quasi-adjacence. Le diam\`etre de $\bar{\D}(\isome{{f_1}}(\l),\isome{{\id}}(\l))$ sera ainsi inf\'erieur ou \'egal \`a $1,5$.

            D'apr\`es la proposition \ref{prop_connexite}, il existe un entier $n\in\N$ et une suite finie $\{c_i\}_{0\leq i\leq n}$ de classes du segment $[\l,\isome{{f_1}}(\l)]$ telle que $c_0=\l$, $c_n=\isome{{f_1}}(\l)$ et \[\bar{\D}_{[\l,\isome{{f_1}}(\l)]}=\underset{1\leq i \leq n}{\bigcup}\bar{\D}_{c_i},\] o\`u l'intersection $\bar{\D}_{c_i}\cap\bar{\D}_{c_{i+1}}$ est non-vide pour tout $0\leq i\leq n-1$.
            Montrons par r\'ecurrence sur $n$ que tous les sommets de $\bar{\D}_{c_i}$ ont la classe $s$ dans leur bord \`a l'infini.
            C'est le cas pour $\isome{\id}(\l)$ et $\bar{\D}_{\l}=\{\isome{\id}(\l)\}$ par \cite[Proposition 5.19]{Lovoronoi1}. Supposons que le r\'esultat soit vrai pour $c_{i-1}$ et montrons-le pour $c_i$.

            Soit $f\in\Bir(\P^2)$ telle que $\isome{f}(\l)\in\bar{\D}_{c_{i-1}}\cap\bar{\D}_{c_i}$. Par hypoth\`ese de r\'ecurrence, la classe $s$ est dans le bord \`a l'infini de la cellule $\V(f)$. D'apr\`es la proposition \cite[Proposition 5.19]{Lovoronoi1} cela implique que $\Bs(f^{-1})$ est inclus dans le support de $s$. En utilisant le lemme \cite[Lemme 5.15]{Lovoronoi1}, nous obtenons que \[\isome{f}^{-1}(s)=3\l-\underset{q\in\Bs(f)}{\sum}e_q-\underset{\substack{p\in\supp(s)\\p\notin\Bs(f)}}{\sum}\isome{f}^{-1}(e_p)\] est une classe $9$-sym\'etrique appartenant au bord \`a l'infini de la cellule $\V(\id)$. Soit $g\in\bar{\D}_{c_i}$. La classe $\isome{f}^{-1}(c_i)$ appartient \`a l'intersection des cellules $\V(\id)\cap \V(f^{-1}\circ g)$.
            De plus, les supports de $c_i$ et de $s$ \'etant les m\^emes il en est de m\^eme des supports de $\isome{f}^{-1}(c_i)$ et de $\isome{f}^{-1}(s)$. Ainsi, d'apr\`es le th\'eor\`eme \cite[Th\'eor\`eme 4.1]{Lovoronoi1} les points-base de $f^{-1}\circ g$ sont inclus dans le support de $\isome{f}^{-1}(s)$, ce qui implique que la classe $\isome{f}^{-1}(s)$ est dans le bord \`a l'infini de la cellule $\V(f^{-1}\circ g)$, par la proposition \cite[Proposition 5.19]{Lovoronoi1}. En appliquant $f$ nous obtenons que la classe $s$ est dans le bord \`a l'infini de la cellule $\V(g)$ comme attendu. Ceci ach\`eve la preuve.
        \end{enumerate}

\vspace{-0.65cm}
    \end{proof}

\bibliographymark{R\'ef\'erences}

\providecommand{\bysame}{\leavevmode\hbox
to3em{\hrulefill}\thinspace}
\providecommand{\arXiv}[2][]{\href{https://arxiv.org/abs/#2}{arXiv:#1#2}}

\providecommand{\og}{``} \providecommand{\fg}{''}
\providecommand{\smfandname}{\&}
\providecommand{\smfedsname}{\'eds.}
\providecommand{\smfedname}{\'ed.}
\providecommand{\smfmastersthesisname}{M\'emoire}
\providecommand{\smfphdthesisname}{Th\`ese}

\end{document}